\documentclass[11pt]{article}
\usepackage{amsfonts}
\usepackage{amsmath,amsthm,amscd,amssymb,mathrsfs,setspace, textcomp}
\usepackage{latexsym,epsf,epsfig}
\usepackage{color}
\usepackage[hmargin=2.25cm,vmargin=2.75cm]{geometry}

\usepackage{hyperref}

\newcommand{\ds}{\displaystyle}

\newcommand{\reals}{\mathbb{R}}

\newcommand{\xb}{{\bf{x}}}

\newcommand{\Dn}{\partial_{\nu}}

\newcommand{\cE}{{\mathcal{E}}}

\newcommand{\cF}{\mathcal{F}}

\newcommand{\R}{\mathbb{R}}

\newcommand{\e}{\epsilon}
\newcommand{\Ez}{E_{z}}

\theoremstyle{plain}
\newtheorem{theorem}{Theorem}[section]
\newtheorem{lemma}[theorem]{Lemma}
\newtheorem{proposition}[theorem]{Proposition}
\newtheorem{condition}{Condition}
\newtheorem{definition}{Definition}
\newtheorem{conjecture}{Conjecture}

\theoremstyle{remark}
\newtheorem{remark}{Remark}[section]
\numberwithin{equation}{section}
\numberwithin{theorem}{section}
\numberwithin{remark}{section}
\numberwithin{assumption}{section}
\numberwithin{condition}{section}

\title{Quasi-stability and Exponential Attractors for A Non-Gradient System---Applications to Piston-Theoretic Plates with Internal Damping}

 \author{\normalsize \begin{tabular}[t]{c@{\extracolsep{.6em}}c@{\extracolsep{.6em}}c}
      Jason Howell &  Irena Lasiecka & Justin T. Webster\footnote{Corresponding author}  \\
\it  College of Charleston & \it Univ. of Memphis    &\it  College of Charleston \\
\it Charleston, SC & \it Memphis, TN &\it Charleston, SC\\
howelljs@cofc.edu &\it lasiecka@memphis.edu&  \it websterj@cofc.edu
\end{tabular}}

\begin{document}
\maketitle
\begin{center}
\large {In memory of Igor D. Chueshov}
\end{center}

\begin{abstract} {\noindent
We consider a nonlinear (Berger or Von Karman) clamped plate model with a {\em piston-theoretic} right hand side---which include non-dissipative, non-conservative lower order terms. The model arises in aeroelasticity when a panel is immersed in a high velocity linear potential flow; in this case the effect of the flow can be captured by a dynamic pressure term written in terms of the material derivative of the plate's displacement. The effect of fully-supported internal damping is studied for both Berger and von Karman dynamics. The non-dissipative nature of the dynamics preclude the use of strong tools such as backward-in-time smallness of velocities and finiteness of the dissipation integral. Modern quasi-stability techniques are utilized to show the existence of compact global attractors and generalized fractal exponential attractors. Specific results depending on the size of the damping parameter and the nonlinearity in force. For the Berger plate, in the presence of large damping, the existence of a proper global attractor (whose fractal dimension is finite in the state space) is shown via a decomposition of the nonlinear dynamics. This leads to the construction of a compact set upon which quasi-stability theory can be implemented. Numerical investigations for  appropriate 1-D models are presented which explore and support the abstract results presented herein.  \\
\noindent {\bf Key terms}:  global attractor, exponential attractor, nonlinear plate equation, quasi-stability
 \\
}
\end{abstract}

\maketitle

\section{Introduction}
In this work we consider a nonlinear plate equation (in the sense of large deflections \cite{bolotin,ciarlet,dowellnon,dowell}) in the absence of rotational inertia and in the presence of a non-conservative, non-dissipative right hand side. This model arises in application studying a flexible plate (or shell) immersed in a high-velocity inviscid potential flow. We consider linear, fully-supported interior damping with scaling given (i) as a contribution from the down-wash of the flow dynamics; and (ii) as a control parameter. Specifically, we continue work initiated in \cite{springer} (and references therein) concerning the long-time behavior of generalized (and weak solutions) in the presence of damping. We note that in \cite{springer} the primary focus is on nonlinear damping; we focus on linear damping for reasons outlined below (and for simplicity), and our primary results depend on the size of the damping.

In particular, we are interested in the existence and properties of compact global attractors, as well as fractal exponential attractors (and generalized fractal exponential attractors) for the dynamical systems generated by generalized solutions to this nonlinear plate equation. (See the Appendix for definitions and standard theorems pertaining to attractors and long-time behavior of dynamical systems; we also note the seminal reference \cite{Babin-Vishik}.) We note that one of the primary objects of interest here is the  {\em generalized} fractal exponential attractor---which is not necessarily the object of interest for other analyses, e.g., \cite{fgmz, kalzel}. The generalized fractal exponential attractor is a compact subset of the state space which: (i) has finite fractal dimension in some space (perhaps with a weaker topology than the state space); and (ii) uniformly attracts bounded subsets of the state space (in the topology of the state space) with exponential rate. When we refer to {\em fractal exponential attractors} we mean something stronger: namely, the same as the above with bounded fractal dimension in the state space itself. 

\subsection{Motivation: Panel Flutter and Flow-Plate Models}
In the the study of the {\em panel flutter} phenomenon in aeroelasticity, one considers a clamped or hinged nonlinear plate (von Karman or Berger) immersed in a potential flow \cite{bolotin,dowell,dowellnon}. A full {\em flow-plate} model is utilized which has the nonlinear plate embedded in the boundary of the flow domain, and the interactive dynamics are given by a wave-plate coupling along the interface. To simplify the analysis, under certain physical assumptions (or in certain parameter regimes---e.g., large flow speeds or low characteristic frequencies), one can {\em reduce} the full flow-plate model to a stand-alone plate model with delay, or with non-dissipative effects. The later case is referred to as piston theory\footnote{The nomenclature ``piston" derives from the assumption the in certain parameter regimes the dynamic fluid pressure acting on the plate can be treated as pressure on the head of piston moving in cylinder.}\cite{pist2,dowell,pist1,jfs, vedeneev} (or, invoking the law of {\em plane sections} \cite{bolotin}). (See \cite{survey1,survey2} for recent overviews on these flow-plate model reductions and references to relevant mathematical and engineering literature.) 

In the analysis of the full flow-plate model one considers the pressure acting on the surface of the plate, $p(x,t)$, as a function of the flow dynamics restricted to the interface. In classical piston theory \cite{bolotin,dowellnon,oldpiston} one restricts to ``valid" parameter regimes and replaces flow effects driving the plate dynamics (the {\em dynamic pressure across the plate})  with a nonlinear function of the so called {\em downwash}: $$P(x,t)=p(x)+D\Big(-[u_t+Uu_x]\Big),$$ where, $D(\cdot)$ can be any number of (potentially nonlinear \cite{pist2,pist1,survey1,survey2}) expressions. However, the most common piston-theoretic model (utilized for the sake of simplicity, and which is sufficiently accurate here \cite{jfs,vedeneev}) is linear.

Thus we arrive at the following model (in the absence of imposed structural damping):
\begin{equation}\label{plate-stand}\begin{cases}
u_{tt}+\Delta^2u+f(u) = p-\mu[u_t+Uu_x] ~~ \text { in } ~\Omega\times (0,T), \\
u=\Dn u =0  ~~\text{ on } ~ \partial\Omega\times (0,T).
\end{cases}
\end{equation}
Here $U>0$ is the {\em unperturbed flow velocity} and $\mu>0$ measures the strength of the interaction between the fluid and the structure (and typically depends on, for instance, free-stream density \cite{vedeneev}). 

We note that the effects of the flow are two-fold and somewhat conflicting: (i) a linear damping term with ``good sign", fully-supported on the interior of the plate, and (ii) a non-dissipative  (and destabilizing) term, scaled by the flow velocity parameter $U$. In some sense the interaction between these two terms is the key driver of interest in this plate problem (and in the general flutter phenomenon). 
Such a non-dissipative, {\em piston-theoretic} model has been intensively studied in 
for different boundary conditions and resistive damping forces, see \cite{daniellorena,Memoires,springer}.

In line with the aeroelasticity literature \cite{dowell,dowellnon} (and references therein), we primarily focus on the theory of large deflections for the plate. Taking the boundary to be clamped (as we do) or hinged, von Karman theory is considered accurate \cite{ciarlet,dowell,lagnese}.  Moreover, as we consider thin plates, we do not consider rotational inertia effects \cite{lagnese}, nor in plane accelerations \cite{lagleug}. In addition, by making an ad hoc simplification to von Karman's scalar plate model, one arrives at a ``friendlier" plate model, known as the {\em Berger equation} \cite{berger,menz2,yang}\footnote{See \cite{gw} for an in-depth discussion of this simplification and its implications.}. Other work on flutter dynamics have considered Berger-like beam and plate models (see \cite{Memoires} and references therein), as well as the scalar von Karman equations \cite{springer,delay,supersonic,webster}. Physically, when the plate has no portion of its boundary {\em free}, the Berger approximation of von Karman's dynamics is taken to be valid \cite{gw,inconsistent,studyberger}.
\begin{remark}
In 1-D, the appropriate model (reflecting an appropriate singular limit of the full von Karman dynamics \cite{lagleug} as in-plane accelerations vanish) is the 1-D analog of the Berger model presented here (sometimes also referred to as a Timoshenko beam \cite{menz}, or Krieger beam \cite{beam1}). In Section \ref{numerics} we consider some numerical results for this 1-D Berger beam model. 
\end{remark}

 Many of the results herein are simultaneously valid for both the von Karman {\em and} Berger plate dynamics. However, the principal results presented below in Section \ref{mainresults} on the existence and structure of limiting dynamics, as well as the appropriate analytical techniques, depend on which nonlinearity is in force. Both nonlinearities are nonlocal and of cubic-type on the finite energy space. We note that, in many cases---not here---the analysis of the clamped scalar von Karman plate entirely subsumes that of Berger---see the abstract analysis in \cite{Memoires,springer}.

\subsection{Model} Let $\Omega \subset \mathbb R^2_{x,y}$ be a domain with smooth boundary $\partial \Omega = \Gamma$. We identify the reference configuration of the plate (of negligible thickness, consistent with the conventions of large deflection theory and aeroelasticity \cite{ciarlet, lagnese}) with $\Omega$. The function $u: \Omega \times \mathbb R_+ \to \mathbb R$ represents the transverse ($z$) displacement of the point $(x,y)$ at the instant $t$. 

To describe the nonlinear elastic dynamics, we consider a  plate equation which is a generalization of that arising in linear {\em piston theory}, as discussed above:
\begin{equation}\label{plate}
\begin{cases}
u_{tt}+\Delta^2u+ku_t+f(u)=p+ Lu & \text{ in } \Omega \\
u=\partial_{\nu}u=0 & \text{ on }~ \Gamma \\
u(0)=u_0,~~u_t(0)=u_1
\end{cases}
\end{equation}
We take clamped boundary conditions corresponding to a panel configuration---which provides some simplification---though we could consider simply-supported, hinged, or various combinations of these conditions\footnote{One should note that changes in the boundary conditions, even on a small portion of the boundary, can lead to dramatic changes in the dynamics and may necessitate an entirely different model.}. 
The linear operator\footnote{In practice we think of $L$ as a first order differential operator with smooth, bounded coefficients.} \begin{equation}\label{el}L: H^{2-\sigma}(\Omega) \to L^2(\Omega),~~\sigma>0,~~\text{ is continuous.}\end{equation} The function $p \in L^2(\Omega)$ represents a static pressure on the surface of the plate.

We consider two dynamics:
\begin{equation*} f(u)=
\begin{cases}
f_B(u)=\Big(b -b_0||\nabla u||_{L^2(\Omega)}^2\Big)\Delta u,&~~\text{Berger}\\[.2cm]
f_V(u)= -[u,v(u)+F_0],&~~\text{von Karman}.
\end{cases}
\end{equation*}
The parameter $b$ is a physical parameter capturing in-plane tension or compression at equilibrium; we take $b>0$, indicative of the presence of forces in the direction of $-\nu(x,y)$ (the unit inward normal). The parameter $b_0$ measures the strength of the nonlinearity, and until Section \ref{numerics} we take $b_0=1$ for simplicity. For $f_V$, the function $F_0 \in H^4(\Omega)$ is representative of in-plane forces. The notation $[\cdot,\cdot]$ corresponds to the von Karman bracket $\ds [u,w] = u_{xx}w_{yy}+u_{yy}w_{xx}-2u_{xy}w_{xy}$, with $v(u)=v(u,u)$ denoting the Airy stress function \cite{springer,ciarlet}, given as the biharmonic solver:
\begin{equation}\label{airy}\begin{cases}
\Delta^2 v(u,w) = -[u,w] ~~~\text{ in }~\Omega \\
\partial_{\nu} v(u,w) = v(u,w) = 0 ~~~\text{ on }~\Gamma,
\end{cases}
\end{equation}
for given $u,w \in H_0^2(\Omega)$. We will use the shorthand $v(u,u)=v(u)$. 

The damping coefficient $k$ can be thought of as representing two distinct phenomenon in the modeling of flow-plate interactions: (i) In the reduction of a full flow-plate interaction to a piston-theoretic plate \cite{pist2,survey1,survey2,pist1,jfs,vedeneev}, the RHS of the plate equation inherits the theoretic term ~$\ds-\mu[u_t+Uu_x],$~ where $U$ corresponds to the (large) unperturbed flow velocity and $\mu$ measures the strength of the interaction between the plate and flow dynamics. One can see that our operator $L$ (non-dissipative) embodies the lower order differential operator ~$-\mu U\partial_x$; due to the {\em sign} of the piston-theoretic term, a natural damping of the form $\mu u_t$ is inherited on the LHS of plate equation, {\em even if there is no structural damping imposed} in the model.
(ii) The damping coefficient $k$ can also incorporate frictional damping (or other viscous damping parameters), or a fully-supported velocity feedback control. We measure these effects via a coefficient $k_0$. 

Thus we can envision the damping parameter $k$ in the model as $k=\mu+k_0$, where $\mu$ is a magnitude given from the properties of the flow model and interaction, and $k_0$ can be thought of as a control or material parameter which can be made large. In what follows, two types of results will be presented on the existence and properties of exponential attractors: (i) results independent of the magnitude of $k$ (so long as $k>0$), and (ii) results dependent on the size of $k>k_{\text{min}}$. As one might guess---and as we will see---type (ii) results have stronger conclusions, owing the stronger damping presence. 

\noindent{\bf  Notation:}
 For the remainder of the text denote $\xb$ for $(x,y) \in \Omega$. Norms $||\cdot||$ are taken to be $L^2(D)$ for the domain $D$. The symbol $\nu$ will be used to denote the unit normal vector to a given domain $D$. Inner products in $L^2(D)$ are written $(\cdot,\cdot)$. Also, $ H^s(D)$ will denote the Sobolev space of order $s$, defined on a domain $D$, and $H^s_0(D)$ denotes the closure of $C_0^{\infty}(D)$ in the $H^s(D)$ norm denoted by $\|\cdot\|_{H^s(D)}$ or $\|\cdot\|_{s,D}$.  We make use of the standard notation for the boundary trace of functions, e.g., for $w \in H^1(D)$, $ \gamma (w) \equiv tr[w]=w \big|_{\partial D}$ is the trace of $w$.
\subsection{Energies and Well-Posedness}\label{en}
We denote the standard plate energy in the non-rotational case \cite{ciarlet,lagnese}:
\begin{align}
E(t) =&~\dfrac{1}{2}\Big[||\Delta u(t)||^2 +||u_t(t)||^2\Big] \\
\mathcal E(t) = &~ E(t)+\Pi(t)
\end{align}
where $\Pi$ represents the nonlinear and unsigned terms:
\begin{equation} \Pi \equiv 
\begin{cases}
\Pi_V = \dfrac{1}{4} ||\Delta v(u)||^2-(F_0,[u,u]) -(p,u) \\[.3cm]
\Pi_B = \dfrac{1}{4} ||\nabla u||^4-\dfrac{b}{2}||\nabla u||^2-(p,u). 
\end{cases}
\end{equation}

We will also delineate positive energies below, which we will denote by:
\begin{equation} E^*(t) \equiv 
\begin{cases}
E^*_{V}(t) = E(t)+\dfrac{1}{4} ||\Delta v(u(t))||^2 \\[.3cm]
E^*_{B}(t) = E(t)+ \dfrac{1}{4} ||\nabla u(t)||^4.
\end{cases}
\end{equation}
When the context is clear we will simply write $E^*(t)$. The topology of finite energy will be denoted by $Y\equiv H_0^2 (\Omega) \times L^2(\Omega)$. This space corresponds to $\mathscr D(A^{1/2})$, where $A: L^2(\Omega) \to L^2(\Omega)$ represents the biharmonic operator with clamped boundary conditions \cite[Chapter 1.3]{springer}. We will also consider a stronger space $W \equiv (H^4\cap H_0^2)(\Omega) \times H_0^2(\Omega)$; we have that $\mathscr D(A) \equiv W$ \cite{springer,LT}.

Here we do not provide a thorough discussion of solutions (weak, generalized, and strong). Rather we refer to the detailed exposition in \cite[Section 2.4]{springer} for abstract second order equations and \cite[Chapter 4]{springer} for the von Karman plate equation (as considered here). We suffice to say that {\em strong solutions} correspond to those solutions with initial data in the domain of the generator for the linear elasticity problem at hand ($W$); {\em generalized solutions} (sometimes called mild) are {strong limits of strong solutions} in the finite-energy topology ($Y$), and correspond to so called semigroup solutions; finally, {\em weak solutions} satisfy a natural variation problem associated to \eqref{plate}.

We have the following theorems concerning the well-posedness of generalized solutions to \eqref{plate} and the corresponding boundedness in time. The following well-posedness result is taken from \cite[Section 4.1.1, p. 197]{springer} for the von Karman dynamics, and \cite{Memoires} for the Berger dynamics.
\begin{proposition}\label{p:well} Let $k\ge0$ and consider either $f_B$ or $f_V$. 
For initial data
$$
(u_0,u_1) \in Y\equiv H_0^2(\Omega)\times L^2(\Omega),
$$
problem (\ref{plate}) has a unique generalized solution on $[0,T]$ for any $T>0$. This solution belongs to
the class
$C\left(0,T;H_0^2(\Omega)\right)\cap C^1\left(0,T;L^2(\Omega)\right),
$
and satisfies the energy identity
\begin{equation}\label{energyrelation}
\mathcal E(t)+k\int_s^t ||u_t(\tau)||^2d\tau=\mathcal E(s)+\int_s^t \big(Lu(\tau),u_t(\tau)\big)_{\Omega} d\tau.
\end{equation}
Moreover, the dynamics $(u,u_t)$ on $Y$ generate a dynamical system $(S_t,Y)$ where for $y_0=(u_0,u_1)$ we have $S_ty_0=(u(t),u_t(t))$ with estimate: \begin{equation}\label{dynsys}
||S_ty_0-S_tx_0||_Y^2\le a_0 e^{\omega_0 t} ||y_0-x_0||_Y^2, \end{equation}
for some $a_0,\omega_0>0$. 
\end{proposition}

The following potential bound  is key in allowing the analysis of long-time behavior for the dynamical system generated by solutions. The result can be found in \cite[p. 49]{springer} for the von Karman dynamics, and \cite[p.163]{Memoires} for the Berger dynamics.
 \begin{proposition}\label{potentiallowerbound}
 For any $u \in H^2(\Omega)\cap H_0^1(\Omega)$ we have for any $\epsilon >0$ and any $0< \eta \le 2$:
 \begin{align}
 ||u||^2_{2-\eta} \le&~ \epsilon \Big(||u||_2^2+||\Delta v(u)||^2\Big) + C_{\epsilon},\hskip1cm\text{ for }~f=f_V\\[.2cm]
 ||u||^2_{2-\eta} \le&~ \epsilon \Big(||u||_2^2+||\nabla u||^4\Big) + C_{\epsilon}, \hskip1.5cm \text{ for }~f=f_B
 \end{align}
 From this we have immediately that there exists $c,C,$ and $M$---which depend on $L$, $p$, and $F_0$ (for the case of $f_V$) or $b$ (for the case of $f_B$):
 \begin{equation}\label{enerequiv}
 cE^*-M \le \mathcal E \le CE^*+M
 \end{equation}
 \end{proposition}
 
 \begin{remark}
 Unlike the gradient case of these nonlinear plate dynamics ($L\equiv 0$), boundedness-in-time of trajectories is not immediately obvious from \eqref{potentiallowerbound} above---as the non-dissipative term is integrated in time. In fact, this is the major issue in dealing with the model at hand.
 \end{remark}

\subsection{Results and Discussion}
The principal results in this treatment concern the existence and properties of attractors and fractal exponential attractors for the dynamical system to by generalized (semigroup) solutions to \eqref{plate}. We address results for both $f_B$ and $f_V$ with arbitrary positive damping coefficient, as well as scaling up the damping coefficient up to improve results. 
\subsubsection{Main Results}\label{mainresults}
The first result here is an improvement over the corresponding one given in \cite{springer}\footnote{The result to which we refer in \cite{springer} is: for the non-gradient dynamics, in the presence of {\em any damping}, a compact global attractor exists.}. We show here that {\em for the non-gradient dynamics ($L \not\equiv 0$)}, {\em any} damping is sufficient to yield the existence of a compact global attractor for the dynamics which is both smooth and finite dimensional. This result holds for both von Karman ($f=f_V$) and Berger ($f=f_B$) dynamics $(S_t,Y)$. 
\begin{theorem}\label{th:main1}
For $f=f_B$ or $f_V$, and {\bf{any}} $k>0$, there exists a compact global attractor $\mathcal A$ for the dynamical system $(S_t,Y)$ corresponding to generalized solutions to \eqref{plate}. The attractor is smooth in the sense of $\mathcal A \subset W$ (bounded in this topology) and has finite fractal dimension in the space $Y$. 
\end{theorem}
The existence of a compact attractor for a the non-gradient system without large damping coefficient $k$ has been shown (abstractly) in \cite[Section 8.3.2]{springer}. However, both smoothness and finite dimensionality of the attractor in this situation required large damping \cite[pp. 424--425]{springer}. Theorem \ref{th:main1} eliminates this requirement. 

Considering the Berger dynamics, $f_B$ in \eqref{plate}, we show the so called {\em quasi-stability} estimate (in the sense of $Y$) on any bounded forward invariant subset of $Y$ (see the Appendix and Section \ref{qsec}). This results in the following theorem.
\begin{theorem}\label{th:main2}
For $f=f_B$ and any $k>0$  there exists a generalized fractal exponential attractor $\widetilde{\mathcal A}_{\text{exp}}\subset Y$ (finite dimension in $\widetilde Y \equiv L^2(\Omega) \times H^{-2}(\Omega)$) for  $(S_t,Y)$. 
\end{theorem}
Theorem \ref{largek1} presented below was shown in \cite[Theorem 9.5.13]{springer}. To obtain the quasi-stability estimate on a bounded forward invariant set when $L\not\equiv 0$ for $f=f_V$, one can take sufficiently large damping. We include this result here for the sake of comparison with the novel results herein.
\begin{theorem}\label{largek1}
For $f=f_V$, when $L \not\equiv 0$ and $k$ sufficiently large (depending on the internal parameters of the problem: $\Omega, L, F_0,p$) there exists a generalized fractal exponential attractor  $\widetilde{\mathcal A}_{\text{exp}}\subset Y$  (finite dimension in $\widetilde Y \equiv L^2(\Omega) \times H^{-2}(\Omega)$) for $(S_t,Y)$.
\end{theorem}
With large damping in the case of Berger's dynamics $f=f_B$ we can improve the generalized fractal exponential attractor to a proper exponential attractor:
\begin{theorem}\label{largek2} For $f=f_B$ and $k$ sufficiently large (depending on the internal parameters of the problem: $\Omega, L, b,p$) there exists a fractal exponential attractor ${\mathcal A}_{\text{exp}} \subset Y$, with finite fractal dimension in $Y$ for $(S_t,Y)$. 
\end{theorem}
Theorem \ref{largek2} is obtained (in part) by operating directly on a smooth set we construct in Section \ref{decomp*}. This set is uniformly exponentially attracting, though we do not have control over its dimensionality (even in the state space $Y$). We state this as an independent theorem here:
\begin{theorem}\label{expset}
For $f=f_B$ and $k$ sufficiently large (depending on the internal parameters of the problem: $\Omega, L, b,p$) there exists a smooth set $\mathcal S$ (smooth in the sense of $\mathcal S \subset W$), which is uniformly exponentially attracting for $(S_t,Y)$. 
\end{theorem}

\subsubsection{Past Results in Relation to Work Herein}
The problem at hand---the long-time behavior of a non-conservative, non-rotational, nonlinear plate---is mathematically and physically compelling. It arises in engineering application in the study of the flutter phenomenon, specifically, the piston-theoretic reduction of a model for a fluttering plate. The analysis is challenging primarily for the reasons described below.

 This is a non-gradient problem, owing to the non-dissipative lower-order term $Lu$. Despite the presence of this term,  using the nonlinear effects in  the system one {\em can} show the existence of a compact global attractor (see \cite[Section 9.4]{springer}). The attractor consists of all bounded full trajectories; however, for such non-gradient dynamics there is no further characterization of the attractor as the unstable manifold of the stationary set. This precludes the ability to use the powerful technique of {\em backward smallness} of velocities to obtain the quasi-stability estimate.

Additionally, the presence of the non-dissipative term destroys the energy relation \eqref{energyrelation} for the dynamics. This immediately clouds the issue of global-in-time boundedness of solutions (and ultimate dissipativity of the associated dynamical system); moreover, the finiteness of the dissipation integral~ $\ds k \int_0^{\infty} ||u_t||^2$ ~is lost, which, although not a central issue here, can impact the analysis of solutions in higher topologies. 

In the absence of rotational inertia effects (in line with the engineering literature) \cite{dowell} the nonlinearities associated with {\em large deflections} are {\em critical}. This is to say they do not act in a compact fashion on the finite energy space. For this reason, modern techniques and compensated compactness methods \cite{springer,kh} must be utilized in the study of long-time behavior. 

Generally, there  are two fundamental properties of  a dynamical system which form the cornerstones of the supporting theory; these are: dissipativity (in the sense of energies, in contrast to {\em ultimate dissipativity}) and smoothing for the dynamics. The property of dissipativity provides a priori bounds, and leads to the construction of well-structured Lyapunov function. Smoothing (a gain of regularity) leads to compactness, and convergence of trajectories. These are the building blocks characterizing long-time behavior analyses. The problem addressed in this work cannot appeal to either of these two properties. 

The loss of a gradient structure for the dynamics, along with insufficient regularity, prevents the  applicability of  standard methods addressing {\em asymptotic smoothness} (see the Appendix).  In order to compensate, one resorts to the following strategy:  to rebound from loss of dissipativity, the nonlinearity plays a critical role (and provides boundedness of orbits). It is precisely the physical, (superlinear) nonlinearity which converts escaping orbits to  bounded, but (potentially) chaotic, orbits. Of course, the  nonlinearity---with its lack of smoothing effects---further contributes to a ``lack of compactness". It is here where ``compensated compactness" is key; in other words, a special structure for the nonlinearity  plays a major role. 
Thus, harvesting ultimate dissipativity from nonlinear effects, and harvesting ``smoothness" from the structure   the of nonlinear terms are the main points in the game. More specific information is included below. 
\begin{itemize}
\item For standard boundary conditions and nonlinear internal damping (under appropriate assumptions) the existence of a compact global attractor is known {\em in the presence of} $L \ne 0$ for {\em any damping coefficient} (this is perhaps the principal theme in \cite{springer}). 
\item If the damping coefficient is large, OR $L \equiv 0$, then the quasi-stability estimate can be shown on the compact global attractor to get finite dimensionality and smoothness of the attractor. The technique from \cite{glw,delay} (showcased here) circumvents the need for backward-in-time-smallness, the key ingredient utilized in \cite{springer} which requires the dynamics to be gradient (i.e., $L\equiv 0$).
\item On \cite[p. 526]{springer}, it is noted that with linear damping and with $L\equiv 0$ the attractor can be shown to be exponential. This appeals to the gradient structure of the problem, but only insofar as to obtain the quasi-stability estimate on the attractor. 
\item \cite[pp. 59--60]{Memoires} deals specifically with the Berger nonlinearity. Existence of a smooth and finite dimensional attractor for the dynamics is shown  only for dissipative dynamics. The latter properties, as in \cite{springer}, are based on proving the quasi-stability on the attractor. However, existence of exponential attractors requires more subtle estimates. These are given in the present paper. 
\item The works \cite{bucci1,bucci2} consider internal damping for a wave-plate system (which includes a Berger plate). In \cite{bucci1} the system is dissipative, and attractors are investigated. In \cite{bucci2}, the dynamics remain dissipative and thermal effects are considered in the plate; here, results on exponential attractors are established.
\end{itemize}
Thus, the primary novel facets of the present treatment are the lack of dissipativity in the dynamics and the lack of smoothing effects.
\begin{remark} 
In this treatment we focus on linear damping, though the technology certainly exists to address nonlinear interior damping. The primary reason for this concerns the fact that, in application, the model of interest has linear damping built-in (the $\mu$ component of the RHS in \eqref{plate-stand}). Secondly, our primary focus is to showcase techniques in the theory of quasi-stability (and their relation to fractal exponential attractors) and a novel decomposition of the dynamics for the Berger plate which leads to the direct construction of a fractal exponential attractor. 
\end{remark}

\subsubsection{Principal Contributions of This Work}
The primary contributions of this treatment are: 

\begin{enumerate}
\item We provide a side-by-side analysis of piston-theoretic Berger and von Karman plates with interior damping (arising in application), and the corresponding study of global attractors and exponential attractors, detailing where the dynamics (and analytical techniques) diverge for the two models.

\item We perform a modern quasi-stability analysis \cite{quasi, springer} on these non-gradient plate dynamics. Our work herein demonstrates the power of the quasi-stability approach; namely, it simplifies and unifies many of the ground-breaking and fundamental studies over the past 30 years on the qualitative properties of dissipative dynamics, (e.g., compact global attractors and exponential attractors). This quasi-stability analysis demonstrates that at the heart of the long-time behavior analysis is a decomposition of the difference dynamics into a stable component and compact (lower order) component. In essence, one estimate (though perhaps on various types of sets) yields an immense amount of ``mileage". We believe the quasi-stability approach to be of great value to others, and this treatment serves as a testament to this. We demonstrate the relevant techniques and showcase the powerful theorems. 

\item We consider the full piston-theoretic model, and provide a modern analysis of compact global attractors for such plate models with their non-dissipative terms. We carefully note the effect of interior damping (its presence {\em and} its size).  Much of the work to date (see the above discussion) on smooth, finite dimensional (and exponential) attractors involves one of two assumptions (see \cite{springer} and references therein): either, one assumes (i) a gradient structure (in order to obtain the quasi-stability estimate or something like it), or (ii) one takes large damping in order to obtain control of non-gradient (non-dissipative) terms. After showing the existence of a compact attractor for solutions, we utilize a recent  finite-mesh technique \cite{delay,glw} to parlay the existence of a compact attractor into the quasi-stability estimate on the attractor. This leads to finite dimensionality and smoothness of the attractor for the von Karman dynamics (\eqref{plate} with $f_V$) without assuming $L\equiv 0$ OR a large damping coefficient $k$.  

\item Additionally, we utilize the techniques in \cite{springer} which lead to generalized fractional exponential attractors via the quasi-stability estimate on an absorbing ball. Specifically, for $f=f_B$ we obtain a generalized fractal exponential attractor with interior damping of arbitrary size; for $f=f_V$ we will require large damping.

\item In the case of the Berger plate $f=f_B$, we additionally show that by incorporating sufficiently large damping, we can directly construct an exponentially attracting set which is smooth. Then via the transitivity property of exponential attraction (Theorem \ref{trans}) and a resulting theorem (Theorem \ref{exp2}), we conclude that the generalized fractal exponential attractor coming from the previous paragraph is actually of finite dimension in the state space, making it a proper fractal exponential attractor. 
\end{enumerate}

\section{Quasi-Stability and Practical Applications}\label{qsec}
Here we define quasi-stability as our primary tool in the long-time behavior analysis.  A quasi-stable dynamical system is one where the difference of two trajectories can be decomposed into a uniformly stable part and a compact part, with controlled scaling of the powers. The theory of quasi-stable dynamical systems has been developed rather thoroughly in recent years \cite{quasi,springer}. This includes more general definitions of quasi-stable dynamical systems \cite{quasi} than what we present below. For ease of exposition and application in our analysis we focus on this more narrow definition. 

Informally, we note that: 
\begin{itemize}
\item Obtaining the quasi-stability estimate on the global attractor $A$ implies additional smoothness and finite dimensionality $A$. This follows from the so called squeezing property and one of Ladyzhenskaya's theorems (see \cite[Theorems 7.3.2 and 7.3.3]{springer}).
\item Obtaining the quasi-stability estimate on an absorbing ball implies the existence of an exponentially attracting set; uniform in time H\"{o}lder continuity (in some topology) yields finite dimensionality of this exponentially attracting set (in said topology). 
\end{itemize}
We now proceed with a formal discussion of quasi-stability.

\begin{condition}\label{secondorder} Consider second order (in time) dynamics $(S_t,H)$ where $H=X \times Z$ with $X,Z$ Banach, and $X$ compactly embeds into $Z$.  Further, suppose $y= (x,z) \in H$ with $S_ty =(x(t),x_t(t))$ where the function $x \in C(\mathbb R_+,X)\cap C^1(\mathbb R_+,Z)$. 
\end{condition}
With Condition \ref{secondorder} we restrict to second order, hyperbolic-like evolutions.
\begin{definition}\label{quasidef}
With Condition \ref{secondorder} in force, suppose that the dynamics $(S_t,H)$ admit the following estimate for $y_1,y_2 \in B \subset H$:
\begin{equation}\label{specquasi}
||S_ty_1-S_ty_2||_H^2 \le e^{-\gamma t}||y_1-y_2||_H^2+C_q\sup_{\tau \in [0,t]} ||x_1-x_2||^2_{Z_*}, ~~\text{ for some }~~\gamma, C_q>0,.
\end{equation} where $Z \subseteq Z_* \subset X$ and the last embedding is compact. Then we say that $(S_t,H)$ is {\em quasi-stable} on $B$.
\end{definition}
\begin{remark}\label{genquas} As mentioned above, the definition of quasi-stability in the key references \cite{quasi,springer} is much more general; specifically, the estimate in \eqref{specquasi} can be replaced with: \begin{equation}\label{genquaseq}
||S_ty_1-S_ty_2||_H^2 \le b(t)||y_1-y_2||_H^2+c(t)\sup_{\tau \in [0,t]} [\mu_H(S_ty_1-S_ty_2)]^2,  \end{equation} 
where: (i) $b(\cdot)$ and $c(\cdot)$ are nonnegative scalar functions on $\mathbb R_+$ such that $c(t)$ is locally bounded on $[0,\infty)$ and $b \in L^1(\mathbb R_+)$ and $\ds \lim_{t \to \infty} b(t) = 0$; (ii) $\mu_H$ is a compact seminorm on $H$. In fact, this definition is recent \cite{quasi}, and is more general than that in \cite{springer}, accommodating a broader class of nonlinear dynamical systems arising the in the long-time analysis of plate models. \end{remark}

We now run through a handful of consequences of the type of quasi-stability described by Definition \ref{quasidef} above for dynamical systems $(S_t,H)$ satisfying Condition \ref{secondorder}.
\cite[Proposition 7.9.4]{springer}
\begin{theorem}[Asymptotic Smoothness]\label{doy}
If a dynamical system $(S_t,H)$ satisfying Condition \ref{secondorder} is quasi-stable on every bounded, forward invariant set $ B \subset H$, then $(S_t,H)$ is asymptotically smooth. Thus, if in addition, $(S_t,H)$ is ultimately dissipative, then by Theorem \ref{dissmooth} there exists a compact global attractor $ A \subset H$. 
\end{theorem}

The theorems in \cite[Theorem 7.9.6 and 7.9.8]{springer} provide the following result concerning improved properties of the attractor $A$ if the quasi-stability estimate can be shown {\em on} $A$.
\begin{theorem}[Dimensionality and Smoothness]\label{dimsmooth}
If a dynamical system $(S_t,H)$ satisfying Condition \ref{secondorder} possesses a compact global attractor $ A \subset H$, and is quasi-stable on $A$, then $ A$ has finite fractal dimension in $H$, i.e., $\text{dim}_f^HA <+\infty$. Moreover, any full trajectory $\{(x(t),x_t(t))~:~t \in \mathbb R\} \subset A$ has the property that
$$x_t \in L^{\infty}(\mathbb R;X)\cap C(\mathbb R;Z);~~x_{tt} \in L^{\infty}(\mathbb R;Z),$$ with bound
$$||x_t(t)||^2_X+||x_{tt}(t)||_Z^2 \le C,$$
where the constant $C$ above depends on the ``compactness constant" $C_q$ in \eqref{specquasi}.
\end{theorem}
\noindent Elliptic regularity can then be applied to the equation itself generating the dynamics $(S_t,H)$ to recover regularity for $x(t)$ in a norm higher than the state space $X$.
\begin{remark}
If quasi-stability is defined as in Remark \ref{genquas}, more assumptions may be needed to infer the additional smoothness from the estimate in \eqref{genquaseq}.
\end{remark}

The following theorem relates generalized fractal exponential attractors to the quasi-stability estimate \cite[p. 388, Theorem 7.9.9]{springer}
\begin{theorem}\label{expattract*}
Let Condition \ref{secondorder} be in force. Assume that the dynamical system generated by solutions $(S_t,H)$ is ultimately dissipative and quasi-stable on a bounded absorbing set $ B$. We also assume there exists a space $\widetilde H \supset H$ so that $t \mapsto S_ty$ is H\"{o}lder continuous in $\widetilde H$ for every $y \in  B$; this is to say there exists $0<\alpha \le 1$ and $C_{ B,T>0}$ so that 
\begin{equation}\label{holder}||S_ty-S_sy||_{\widetilde H} \le C_{ B,T}|t-s|^{\alpha}, ~~t,s\in[0,T],~~y \in  B.\end{equation} Then the dynamical system $(S_t,H)$ possesses a generalized fractal exponential attractor $A_{\text{exp}}$ whose dimension is finite in the space $\widetilde H$, i.e., $\text{dim}_f^{\widetilde H} A_\text{exp}<+\infty$. 
\end{theorem}
\begin{remark}
Remark 7.9.10 \cite[pg. 389]{springer} discusses the need for the H\"{o}lder continuity assumption above. It is presently an open question as to how ``necessary" this condition is for general hyperbolic systems possessing global compact attractors. \end{remark}
\begin{remark}
In addition, owing to the abstract construction of the set $A_{\text{exp}} \subset X$, {\em boundedness} of $A_{\text{exp}}$ in any higher topology is not addressed by Theorem \ref{expattract*}. \end{remark}

The proofs of Theorems \ref{dimsmooth} and \ref{expattract*} can be found in \cite{quasi,springer}, and rely fundamentally on the technique of ``short" trajectories or ``l" trajectories (see, e.g.,  \cite{ltraj2}).

The above two theorems appeal to the quasi-stability property of the dynamics on the global attractor $A$ or the absorbing set $B$. If one can construct a compact set $K$ which is itself exponentially attracting, then having the quasi-stability estimate on $K$ (along with the transitivity of exponential attraction Theorem \ref{trans}) yields a stronger result. This result is given as \cite[Theorem 3.4.8, p. 133]{quasi} and proved there.

\begin{theorem}[\cite{quasi}]\label{exp2}
Let $(S_t,H)$ be a dynamical system, where $H$ is a separable Banach space. Assume:
\begin{enumerate}
\item There exists a positively invariant compact set $F \subset H$ and positive constants $C$ and $\gamma$ such that $$\sup_{t}\left\{d_{H}(S_tx,F)~:~x \in D \right\} \le C e^{-\gamma(t-t_D)},$$ for every bounded set $D \subset H$ and for $t \ge t_D$.
\item There exists a neighborhood $\mathscr O$ of $F$ and numbers $\Delta_1$ and $\alpha_1$ such that 
$$||S_tx_1-S_tx_2|| \le \Delta_1e^{\alpha_1 t}||x_1-x_2||.$$
\item The system $(S_t, H)$ is quasi-stable on $F$ for $t \ge t_*$ for some $t_*>0$. 
\item The mapping $t \mapsto S_tx$ is uniformly H\"{o}lder continuous on $F$; that is there exists constants $C_F(T)>0$ and $\eta \in (0,1]$ such that 
$$||S_{t_1}x-S_{t_2}x|| \le C_F(T)|t_1-t_2|^{\eta},~~t_i \in [0,T],~~x \in F.$$
\end{enumerate}
Then there exists a fractal exponential attractor $A_{\text{exp}} \subset H$ for $(S_t,H)$ whose fractal dimension is finite in $H$.
\end{theorem}

The proof of Theorem \ref{exp2} (as given in \cite{quasi}) proceeds, essentially, as a synthesis of our Theorem \ref{expattract*} and Theorem \ref{trans}. The results of Theorem \ref{expattract*} are applied on the absorbing ball, as well as on the smooth, invariant set $F$. Then transitivity of exponential attraction is invoked for arbitrary bounded subsets of the state space.

\section{The Difference of Trajectories}

In this section we recall a fundamental multiplier inequality, as well as key decompositions of nonlinear terms, to be used in the analysis later. 
We will reference the following difference system, where $u^i$ satisfy \eqref{plate} and $z=u^1-u^2$ and $\mathcal F(z) = f(u^1)-f(u^2)$:
\begin{equation}\label{difference}
\begin{cases}
z_{tt}+\Delta^2z+kz_t+\cF(z) = Lz \\
z=\partial_{\nu} z = 0 \\
z(0)=u^1_0-u^2_0;~~z_t(0)=u^1_1-u^2_1.
\end{cases}
\end{equation}
In reference to \eqref{difference} we will utilize the standard energy on the difference:
\begin{equation}\label{Ez}
E_z(t)=\dfrac{1}{2}\big[||z_t||^2+||\Delta z||^2\big].
\end{equation}
We note the following identities (arrived at first on strong solutions, and then via limit passage on generalized and weak solutions) corresponding to \eqref{difference}. The first is the energy identity, and the second is arrived at via the equipartition multiplier:
\begin{align}\label{stufff}
E_z(t) + k\int_s^t||z_t||^2 =&~-\int_s^t\big(\cF(z),z_t\big)_{\Omega} +\int_s^t\big(Lz,z_t\big)_{\Omega}\\[.2cm]
\int_s^t ||\Delta z||^2 - \int_s^t ||z_t||^2 =&~ \dfrac{k}{2}||z||^2\Big|_s^t+\int_s^t\big(Lz,z\big)_{\Omega} -\int_s^t\big(\cF(z),z\big)_{\Omega}
\end{align}
 
 The following lemma can be seen as a special case of \cite[Lemma 8.3.1, p.398]{springer}. It is a standard estimate (at least for fully-supported interior damping) utilizing \eqref{stufff} above with $k>0$, and it makes use of the locally-Lipschitz property of both $f_V$ and $f_B$ from $H_0^2(\Omega) \to L^2(\Omega)$ (for the latter, see, for instance, \cite[Section 3.5.1]{supersonic}). 

\begin{lemma}\label{le:observbl} 
Let $u^i \in C(0,T;H_0^2(\Omega))\cap C^1(0,T;L^2(\Omega)) $ solve (\ref{plate}) with clamped boundary conditions and appropriate initial conditions on $[0,T]$ for $i=1,2$, $T\ge T^*$. Additionally, assume $u^i(t) \in \mathscr B_R(H^2(\Omega))$ for all $t\in [0,T]$. Then the following estimate holds for $f_B$ or $f_V$ with $a_i$ independent of $T$ and $R$:
\begin{multline}\label{enest1}
  T\Ez(T)+ \int_{0}^T \Ez(\tau) d\tau \le  a_0\Ez(0)
  +C(T,R)\sup_{\tau \in [0,T]}||z||^2_{2-\eta} \\-a_1\int_0^T\int_s^T \big( \cF(z),z_t\big)_{\Omega} d\tau ds  -a_2\int_0^T \big( \cF(z),z_t\big)_{\Omega} ds.
\end{multline}
\end{lemma}
We will utilize (in several places) key decompositions of the term $(\cF(z),z_t)_{\Omega}$ for both $f_B$ and $f_V$. Analysis of these terms differs, owing to the structural differences of the nonlinearities. Perhaps unsurprisingly (as $f_B$ is a simplification of $f_V$), the decomposition of $(\cF(z),z_t)$ is ``friendlier" in the case of the Berger plate---and we exploit this. 
The results stated in the following theorem can be found in \cite[Section 1.4]{springer} for the von Karman dynamics; see \cite{bucci1,bucci2,Memoires,gw} for the Berger dynamics.
\begin{theorem}\label{nonest}
Let $u^i \in \mathscr{B}_R(H^2_0(\Omega))$, $i=1,2$, and $z=u^1-u^2$ with $\cF(z)=f(u^1)-f(u^2)$.

 Then  for either $f=f_V$ or $f=f_B$ we have:
\begin{equation}\label{f-est-lip}
||f(u^1)-f(u^2)||_{-\delta}  \le C_{\delta}\Big(||u^1||_2,||u^2||_2\Big)||z||_{2-\delta} \le C(\delta,R)||z||_{2-\delta},~~\forall~\delta \in [0,1].
\end{equation}
\begin{remark}
When $\delta > 0$ the above property has been known \cite{ciarlet,lagnese}. However, when $\delta =0$ , in the case of the von Karman nonlinearity $f=f_V$, the above inequality follows from ``sharp" regularity of Airy's stress function \cite[p. 44]{springer}.
\end{remark}

\noindent {\bf Von Karman}: If we further assume that $u^i \in C(s,t;H^2(\Omega))\cap C^1(s,t;L^2(\Omega))$,  then taking \newline $\ds f(u)=f_V(u) = - [u,v(u)+F_0]$ ~we have:
\begin{equation}\label{4.9a}
- \big( \cF(z),z_{t}\big)_{\Omega} =\dfrac{1}{4}\frac{d}{dt}Q_0(z)+\frac{1}{2}
P_0(z)
\end{equation}
where
\begin{equation}
Q_0(z)=\big( v(u^1)+v(u^2),[z,z]\big)_{\Omega} -||\Delta v(u^1+u^2,z)||^2
\end{equation}
and
\begin{equation}\label{4.9aa}
P_0(z)=-\big( u^1_{t},[u^1,v(z)]\big)_{\Omega} -\big( u^2_{t},[u^2,v(z)]\big)_{\Omega} -\big(
u^1_{t}+u^2_{t},[z,v(u^1+u^2,z)]\big)_{\Omega}.
\end{equation}
Moreover,
\begin{align}\label{eq4.5}
\left|\int_s^t \big( \cF(z),z_t\big)_{\Omega} d\tau\right| \le &~C(R)\sup_{\tau \in [s,t]} ||z||^2_{2-\eta}+\frac C 2 \left|\int_s^tP_0(z)d\tau\right|
\end{align} for some $0<\eta<1/2$, provided
 $u^i(\tau) \in \mathscr{B}_R(H^2_0(\Omega))$ for all $\tau\in [s,t].$
 \vskip.2cm
\noindent {\bf Berger}: For $u^1,u^2 \in
C(s,t;(H^2\cap H_0^1)(\Omega))\cap C^1(s,t;L^2(\Omega))$, then taking \newline$\ds f(u)=f_B(u) = (b-||\nabla u||^2)\Delta u$~ we have:
\begin{equation*}
- \big( \cF(z),z_{t}\big)_{\Omega} =\dfrac{1}{2}\frac{d}{dt}Q_1(z)+P_1(z)
\end{equation*}
where
\begin{equation*}
Q_1(z)=\big( (b-||\nabla u^1||^2)||\nabla z||^2\big)
\end{equation*}
and
\begin{equation}\label{p1}
P_1(z)=- (\Delta u^1,u^2_t)||\nabla z||^2+\big(||\nabla u^1||^2-||\nabla u^2||^2\big)(\Delta u^2,z_t).
\end{equation}
We have then that, for  $0<\eta<1/2$:
\begin{align}\label{mmmbop}
\Big|\int_s^t \big( \cF(z),z_t\big)_{\Omega} d\tau\Big| \le &~C(R,\epsilon)\sup_{\tau \in [s,t]} ||z||^2_{2-\eta}+ \epsilon\int_s^tE(t)d\tau,~~\forall~~\epsilon>0,
\end{align} provided
 $u^i(\tau) \in \mathscr{B}_R(H^2_0(\Omega))$ for all $\tau\in [s,t].$
\end{theorem}

\section{ Ultimate Dissipativity of the Dynamical System $(S_t,Y)$}\label{dissp}
 Here the structure of the calculations permit us to consider $f_B$ and $f_V$ simultaneously in showing ultimate dissipativity of $(S_t,Y)$. In the sections which follow, we will need to split the analyses to focus on specific properties of $f_B$ and $f_V$.
\begin{theorem}\label{disp} For $f=f_V$ or $f_B$ and any $k>0$, the dynamical system $(S_t,Y)$ is ultimately dissipative. 
\end{theorem}
\begin{proof}[Proof of Theorem \ref{disp}]
Let $k>0$. We consider the following parametric Lyapunov function:
\begin{align}
V(S_t y) \equiv &~\cE(u(t),u_t(t))+\nu\Big(( u_t,u) +\frac{k}{2}||u||^2\Big)
\end{align}
where $S_t y \equiv y(t)= (u(t),u_t(t))$ for $t \ge 0$, and $\nu$ is some small positive number to be determined below.
From Proposition \ref{potentiallowerbound} we have
\begin{equation}\label{energybounds}
c_0E^*(u,u_t) - c \le V(S_ty) \le c_1E^*(u,u_t) +c
\end{equation}
for all $0<\nu<\nu_0$ and $\nu_0$ sufficiently small. Here, $c_0,c_1,c >0$ are constant. The terms $c_0$ and $c$ may depend on $\nu_0$, but do not depend on the damping parameter $k$. 

We now compute $\ds \dfrac{d}{dt} V(S_ty)$:
\begin{align}
\dfrac{d}{dt} V(S_ty)=&\nonumber~\dfrac{d}{dt}\cE(t)+\nu(u_{tt}+ku_t,u)_{\Omega} +\nu||u_t||^2 \\\nonumber
&
\end{align}
We make use of the relation $$\ds u_{tt}+ku_t=-\Delta^2u+p-f(u)+Lu$$ owing to (\ref{plate}). Substituting this into the relation above and simplifying yields:
\begin{align*}
\dfrac{d}{dt} V(S_ty)
=&~ (\nu-k)||u_t||^2-\nu ||\Delta u||^2-\nu(f(u),u) \\
&+\nu(p,u)+(Lu,u_t)+\nu (Lu,u).
\end{align*}
We note that: \begin{align*}
(f_V(u),u)_{\Omega} = & \big(-[u,v(u)+F_0],u\big)_{\Omega} = ||\Delta v(u)||^2-([u,u],F_0) \\[.2cm]
(f_B(u),u)_{\Omega} = & \big((b-||\nabla u||^2)\Delta u, u \big)_{\Omega} = ||\nabla u||^4-b ||\nabla u||^2.
\end{align*}
Additionally, by the assumption \eqref{el} on the ``lower order" operator $L$, we have 
\begin{equation}\label{epsilonlemma}
||Lu||^2 \le ||u||^2_{2-\sigma} \le \epsilon E^*+C_{\epsilon}.\end{equation}
Then, using  (i) Young's inequality, (ii) the bound in Proposition~\ref{potentiallowerbound}, (iii) and by taking $\nu$ sufficiently small, we have lemma that follows.
\begin{remark} We note the role that nonlinear forces play in proving ultimate dissipativity here. Indeed, Lemma \ref{epsilonlemma} is due to superlinear behavior of the Berger and von Karman nonlinearities. 
\end{remark}
 \begin{lemma}\label{le:48}
For \textit{any} $k>0$ there exist $\nu>0$, and $c(\nu,k)>0$, $C_1(\nu,p,F_0),$ and $C_2(\nu,p,b)>0$, such that
\begin{equation}\label{goodneg}
\dfrac{d}{dt}V(S_ty)\le-c\Big\{||u_t||^2+||\Delta u||^2+||\Delta v(u)||^2 \Big\}+C_1,
\end{equation} for $f=f_V$. 

For $f=f_B$ we have
\begin{equation}\label{goodneg1}
\dfrac{d}{dt}V(S_ty)\le-c\Big\{||u_t||^2+||\Delta u||^2+||\nabla u||^4 \Big\}+C_2.
\end{equation}
We have that $c(k) \to +\infty$ as $k \to +\infty$.
\end{lemma}

From this lemma and the upper bound in \eqref{energybounds}, we have for some $\delta(\nu,k)>0$ (with $\delta \to +\infty$ as $k\to +\infty$)  and a $C$ (independent of $k$): \begin{equation}\label{gronish}
\dfrac{d}{dt}V(S_ty) +\delta V(S_ty) \le C,~~t>0.
\end{equation}
The estimate above in (\ref{gronish}) implies (by a version of Gronwall's inequality) that
\begin{equation*}
V(S_ty) \le V(y)e^{-\delta t}+\dfrac{C}{\delta}(1-e^{-\delta t}).
\end{equation*}
Hence, the set
$$
\mathcal{B}_{\delta} \equiv \left\{y \in Y:~V(y) \le 1+\dfrac{C}{\delta} \right\},
$$  is a bounded forward invariant absorbing set. This gives that $(S_t,Y)$ is ultimately dissipative.
\end{proof}
\begin{remark} {\rm
It is clear that if the damping coefficient $k$ is increased, the size of the absorbing set $\mathcal B_{\delta}$ does not increase. We can choose a fixed absorbing set $\mathcal B$, established by some value $0<k_*<k$. It is clear that increasing the damping coefficient will decrease the time of absorption for a fixed absorbing set $\mathcal B$.}
\end{remark}

\section{Analysis of the von Karman Plate: $f=f_V$}
In this section we establish the existence of a compact global attractor for $(S_t,Y)$ in the case of $f=f_V$. We make use of Theorem \ref{dissmooth}, since the dynamics here are non-gradient (due to the presence of the piston-theoretic presence of $L$). We provide the argument below for completeness, though the existence of the attractor is established in \cite{Memoires,springer}. 
\subsection{Existence of a Compact Global Attractor}
Since the dynamical system $(S_t,Y)$ is ultimately dissipative by Theorem \ref{disp}, we now show the {\em asymptotic smoothness} property of the dynamical system in \eqref{plate} taken with $f_V$. We make use of the criterion for asymptotic smoothness given in \cite{springer} (which first appeared in another form in \cite{kh}); it is given in the Appendix as Theorem \ref{psi}. The estimate below follows from \eqref{enest1} by taking $T$ sufficiently large.
\begin{lemma}\label{le:khan}
Let $f=f_V$ in \eqref{plate} and take $k>0$. Suppose $z=u^1-u^2$ is as in (\ref{difference}), with $y^i(t)=(u^i(t),u^i_t(t))$ and $y^i(t) \in \mathscr B_R(Y)$ for all $t\ge 0$. Also, let $\eta >0$ and $\Ez(t)$ be defined as  in (\ref{Ez}).
Then for every $0<\e<1$ there exists  $T=T_\e(R)$ such that the following estimate holds:
\begin{equation*}
E_z(T)  \le \epsilon + \Psi_{\epsilon,T,R}(y^1,y^2),
\end{equation*}
\begin{multline*} \Psi_{\epsilon,T,R}(y^1,y^2) \equiv C(R,T) \sup_{\tau \in [0,T]} ||z(\tau)||_{2-\eta}^2 +a_1\left| \int_0^T\big( \cF(z),z_t\big)_{\Omega} d\tau\right| \\+ a_2\left|\int_0^T \int_s^T \big( \cF(z(\tau)),z_t(\tau)\big)_{\Omega} d\tau ds\right|.\end{multline*}
\end{lemma}
In Lemma \ref{le:khan} above, we have the necessary estimate for asymptotic smoothness; it now suffices to show that $\Psi$, as defined above, has the desired compensated compactness condition.
Before proceeding, let us introduce some notation which will be used throughout the remainder of this section and in the following section. We will write \begin{equation}\label{notations} l.o.t. = \sup_{\tau \in [0,T]}||z(\tau)||^2_{2-\eta}.\end{equation}

\begin{theorem}\label{smoothness}
The dynamical system $(S_t,Y)$ generated by weak solutions to (\ref{plate}) is asymptotically smooth.
\end{theorem}
\begin{proof}[Proof of Theorem \ref{smoothness}]
Let $B$ be a bounded, positively invariant set in $Y$, and let $\{y^n\}\subset B \subset \mathscr{B}_R(Y)$.   We now write $\Psi_{\epsilon,T,R}$ as $\Psi$, with $\epsilon, T,$ and $R$ fixed along with the other constants given by the equation. We would like to show that
$$
\liminf_m \liminf_n \Psi(y^n,y^m) = 0.
$$
More specifically, for any initial data $y_0^n=(u^n_0,u^n_1) \in B $ we define
\begin{equation}
\widetilde{ \Psi} (y^n_0, y^m_0) = \left|\int_0^T \big( \mathcal{F}(z^{n,m}),
z^{n,m}_t \big)_{\Omega} d\tau\right| +\left |\int_0^T \int_s^T \big( \mathcal{F}(z^{n,m}(\tau)) ,z^{n,m}_t(\tau) \big)_{\Omega} d\tau ds \right|
\end{equation}
where  $(z^{n,m},z^{n,m}_t) = (u^n -u^m, u^n_t-u^m_t)$ has initial data $y^n_0-y^m_0$ and solves (\ref{difference}). The key to
compensated compactness is the following representation for the bracket \cite[pp. 598-599]{springer}, which differs from the representation given earlier in Theorem \ref{nonest}:
\begin{align*}
\big( \mathcal{F}(z^{n,m}),z^{n,m}_t\big)_{\Omega} =& ~\frac{1}{4} \frac{d}{d\tau} \Big\{ - ||\Delta v(u^n) ||^2
- ||\Delta v(u^m) ||^2 + 2 \big( [z^{n,m},z^{n,m}],F_0\big)_{\Omega} \Big\}\\\nonumber&-\big( [ v(u^m),u^m], u^n_t\big)_{\Omega} -
\big( [ v(u^n) , u^n ], u^m_t\big)_{\Omega}.
\end{align*}
Integrating the above expression in time and evaluating, with $u^{i} \rightharpoonup u$, yields:
\begin{multline}\label{itlim}
\lim_{n\to \infty}\lim_{m\to \infty} \int_s^T \big( \mathcal{F}(z^{n,m}(\tau))
,z_t^{n,m}(\tau) \big)_{\Omega} d\tau \\ 
=\dfrac{1}{2} \Big\{ ||\Delta v(u(s)) ||^2 - ||\Delta v(u(T))||^2 \Big\} \hspace*{1.5in} \\
- \lim_{n \to \infty} \lim_{m \to \infty} \int_s^T \Big\{ \big( [ v(u^n),
u^n ], u_t^m\big)_{\Omega} + \big( [ v(u^m), u^m ], u_t^n\big)_{\Omega}\Big\},
\end{multline}
where we have used (i) the weak convergence in $H^2(\Omega)$ of $z^{n,m} $
to 0, and (ii) compactness of $\Delta v(\cdot) $ from $H^2(\Omega) \rightarrow
L^2(\Omega)$ \cite[pp. 43--44]{springer}. The iterated limit in (\ref{itlim}) is handled via iterated weak
convergence, as follows:
\begin{multline*}
\lim_{n \to \infty} \lim_{m \to \infty} \int_s^T \Big\{ \big( [ v(u^n), u^n
], u_t^m\big)_{\Omega} + \big( [ v(u^m), u^m ], u_t^n\big)_{\Omega} \Big\}
\\
= 2 \int_s^T \big( [ v(u), u] , u_t\big)_{\Omega} = \frac{1}{2} ||\Delta v(u)(s) ||^2 - \frac{1}{2} || \Delta v(u)(T) ||^2.
\end{multline*}
 This yields the desired conclusion, that
\begin{equation*}
\lim_{n\rightarrow \infty }\lim_{m\rightarrow \infty }\int_{s}^{T}\big(\mathcal{F}( z^{n,m}(\tau)),z_{t}^{n,m}(\tau)\big)_{\Omega} d\tau=0.
\end{equation*}
The second integral term in $\widetilde{ \Psi }$ is handled similarly. Finally, since the term $l.o.t.$ above is
compact (below energy level) via the Sobolev embeddings,
we obtain
\begin{equation*}
\underset{m\rightarrow \infty }{\lim \inf }\, \underset{n\rightarrow \infty }{\lim \inf }~{ \Psi }(y_0^{n},y_0^{m})=0.
\end{equation*}
Invoking Theorem \ref{psi}  concludes the proof of the asymptotic smoothness of $(S_t,Y)$.
\end{proof}

Having shown the asymptotic smoothness property, we can now conclude by Theorem \ref{dissmooth} that there exists a compact global attractor $\mathcal A \subset Y$ for the dynamical system $(S_t, Y)$ when $f=f_V$ in \eqref{plate}.

\subsection{Quasi-stability on $\mathcal A$}\label{refone}
In order to establish both smoothness of the attractor and finite fractal
dimensionality, a stronger estimate on the difference of two trajectories on the attractor is needed. 

In this section we refine our methods in the asymptotic smoothness calculation and work on trajectories from the attractor, whose existence has been established previously. We will take a different approach than in \cite{springer} when addressing the von Karman dynamics.
\begin{lemma}\label{est*}
Let $f=f_V$ and take $k>0$. Suppose $z=u^1-u^2$ is as in (\ref{difference}), with $y^i(t)=(u^i(t),u_t(t)^i)$ and $y^i(t) \in \mathcal A$ for all $t\ge 0$. Also, let $\eta >0$ and $\Ez(t)$ be defined as  in (\ref{Ez}).
Then there exists a time $T$ such that the following estimate holds:
\begin{equation*}
E_{z}(T)+\int_{0}^{T}E_z(\tau) d\tau\leq \alpha E_z(0)+C(\mathcal{A},T,k)\underset{\tau \in \lbrack0,T]}{\sup }||z(\tau )||_{2-\eta }^2,\hskip.5cm \alpha<1.
\end{equation*}
\end{lemma}
The quasi-stability inequality has become a standard tool in proving both smoothness and finite dimensionality of attractors. 
The typical way of proving it  (in the case of ``non-compact" dynamics) is by utilizing the gradient property of the dynamics and running estimates for large negative times. This avenue is not possible in the non-gradient scenario, as the attractor is not characterized by trajectories connecting stationary points. In what follows we expound upon a recent method, which exploits the just-obtained compactness of the attractor, along with an associated $\epsilon$-net property.
\begin{proof}[Proof of Lemma \ref{est*}]
Analyzing \eqref{enest1}, we may also write
\begin{multline}\label{eq-obsr1}
TE_z(T)+\int_0^TE_z(\tau)d\tau \\ \le cE_z(0)+
C\cdot T \sup_{s\in [0,T]}\Big|\int_s^T( \cF(z),z_t)_{\Omega} d\tau \Big|+C(R,T)\sup_{\tau \in [s,t]} ||z||^2_{2-\eta},
\end{multline}
where $\cF(z)$ is given in \eqref{notations}.
We note that $c$ does not depend on $T$.

In order to prove the quasi-stability estimate (as in \eqref{specquasi}), we have to address the non-compact term $(
\mathcal{F}(z),z_{t})_{\Omega}$ using only the established properties of the global attractor: namely, {\em compactness} of $\mathcal A$ in $H_0^2(\Omega)\times L^2(\Omega)$. We recall the  relation \eqref{4.9aa}:
if  $u^i \in C(s,t;H^2(\Omega))\cap C^1(s,t;L^2(\Omega))$ with $u^i(\tau) \in \mathscr{B}_R(H_0^2(\Omega))$
for $\tau\in [s,t]$, then
\begin{align}\label{1new}
\left|
\int_s^t (\cF(z),z_t(\tau))_{\Omega} d\tau\right| \le &~C(R)\sup_{\tau \in [s,t]} ||z||^2_{2-\eta}+\frac C2\left|
 \int_s^t P_0(z(\tau))d\tau\right|
 \end{align} for some $0<\eta<1/2$.
Here $P_0(z)$ is given by (\ref{4.9aa}).
\par
Let $\gamma_{u^1} =\{(u^1(t),u^1_{t}(t)):t\in \mathbb{R\}}$ and $\gamma_{u^2} =\{(u^2(t),u^2_{t}(t)):t\in \mathbb{R\}}$ be trajectories from the
attractor $\mathcal{A}$. It is clear that for the pair $
u^1(t)$ and $u^2(t)$ satisfy the hypotheses of the estimate in \eqref{1new} for every
interval $[s,t]$. Our main goal is to handle the term on the right hand side of (\ref{1new}) involving $P_0$ which is of {\it critical
regularity}. Note that for every $\tau \in \mathbb{R}$, the element $u^i_{t}(\cdot)$ belongs to
a compact set in $L_{2}(\Omega )$. Hence, by density of $H_0^2(\Omega) $ in $L^2(\Omega) $
we can assume, without a loss of generality, that  for every $\epsilon >0$ there exists a finite set $
\{\phi _{j}\}\subset H_{0}^{2}(\Omega )$, $ j = 1,2,...,n(\epsilon) $, such that  for all $\tau \in \mathbb{R} $ we can find indices $
j_{1}(\tau) $ and $j_{2} (\tau) $ with
\begin{equation*}
||u^1_{t}(\tau )-\phi _{j_{1}(\tau)}||+||u^2_{t}(\tau)-\phi _{j_{2}(\tau)}||\leq \epsilon
~~\mbox{ for all } ~\tau \in \mathbb{R}.
\end{equation*}
With $P_0(z)$ given by (\ref{4.9aa}), denote
\begin{equation*}
P_{j_{1},j_{2}}(z)\equiv -\left( \phi _{j_{1}},[u^1,v(z)]\right)_{\Omega} -\left( \phi
_{j_{2}},[u^2,v(z,z)]\right)_{\Omega} -\left( \phi _{j_{1}}+\phi
_{j_{2}},[z,v(u^1+u^2,z)]\right)_{\Omega},
\end{equation*}
where $z(t)=u^1(t)-u^2(t).$ It is easily demonstrated that for all $j_1, j_2 \leq n(\epsilon)$:
\begin{equation}\label{3new}
\big|\big|P_0(z(\tau))-P_{j_{1}(\tau),j_{2}(\tau)}(z(\tau))\big|\big|\leq \epsilon C(\mathcal A)||z(\tau )||_{2}^{2}
\end{equation}
uniformly in $\tau \in  \mathbb{R} $.
\par
Appealing to the estimate  (1.4.17) in \cite[p. 41]{springer},
\begin{equation*}
\big|\big|[u,w]\big|\big|_{-2} \leq C ||u||_{2-\beta}~||w||_{1+\beta},~~  \forall \beta \in [0,1)
\end{equation*}
and exploiting elliptic regularity,
one obtains
\begin{equation}\label{kar2}
\big|\big|[u,v(z,w)]\big|\big|_{-2}  \leq C ||u||_{2-\beta}~\big|\big|[z,w]\big|\big|_{-2}  \leq C ||u||_{2-\beta}~ ||z||_{2-\beta_1}~ ||w||_{1+\beta_1 },
 \end{equation}
 where above inequality holds for any  $\beta, \beta_1 \in [0,1)$. Recalling the  additional smoothness of $\phi_j \in H_0^2(\Omega)$, along with the  estimate in (\ref{kar2})
 applied with $\beta = \beta_1 =\eta $,
 and accounting  the structure of the $P_j$ terms, one obtains immediately:
\begin{equation*}
\big|\big|P_{j_{1},j_{2}}(z)\big|\big|\leq C(\mathcal A)\Big(||\phi _{j_{1}}||_{2}+||\phi
_{j_{2}}||_{2}\Big) ||z(\tau )||_{2-\eta }^{2}
\end{equation*}
for some $0<\eta<1$. So we have
\begin{equation}\label{5.5}
\underset{j_{1},j_{2}}{\sup }||P_{j_{1},j_{2}}(z)||\leq
C(\epsilon)||z(\tau )||_{2-\eta }^{2}~~~\mbox{for some $0<\eta<1$},
\end{equation}
where $C(\epsilon) \rightarrow \infty $ when $\epsilon \rightarrow 0 $.
Taking into account (\ref{3new}) and (\ref{5.5}) in (\ref{1new}):
\begin{equation}\label{4.12}
\left|\int_{s}^{t}(\mathcal{F}
(z),z_{t})_{\Omega} \right|\leq C(\epsilon ,T,\mathcal A)\underset{\tau \in
\lbrack s,t]}{\sup }||z(\tau)||_{2-\eta }^{2}+\epsilon
\int_{s}^{t}||z(\tau )||_{2}^{2}d\tau
\end{equation}
for all $s\in \mathbb{R}$, with $\eta >0$ and $t>s$.
Considering \eqref{4.12}, and taking $T$ fixed but sufficiently large, we have
from \eqref{eq-obsr1} and \eqref{1new} that:
\begin{equation*}
E_{z}(T)\leq \alpha E_z(0)+C\underset{\tau \in \lbrack0,T]}{\sup }||z(\tau )||_{2-\eta }^2, \hskip.5cm \alpha<1.
\end{equation*}

\end{proof}
 A (by now standard) stabilization argument yields that, for a trajectory of differences $(z(t),z_t(t))$,
\begin{equation*}
||(z(t),z_t(t)||_{Y}^2 \le C(\sigma,\mathcal A)||(z_0,z_1)||_{Y}^2e^{-\sigma t}+C^* \sup_{\tau \in [0,t]} ||z(\tau)||_{2-\eta}^2.
\end{equation*}
This is the abstract estimate in Theorem \ref{specquasi}---the quasi-stability estimate. Hence we conclude that the dynamical system $(S_t,Y)$ is quasi-stable on $\mathcal A$. 
\begin{remark} It is important to note that the technique employed here provides a constant $C^*$ above which does not depend on a particular trajectory; thus, in what follows the corresponding estimates are uniform on the absorbing ball---giving boundedness of the attractor in the higher topology. This is a point of departure from the technique which utilizes of {\em backward-in-time smallness} (as in \cite{springer}) to obtain the quasi-stability estimate. \end{remark}

\subsection{Smoothness and Finite Dimensionality of $\mathcal A$---Completion of the Proof of Theorem \ref{th:main1}}\label{smuth}
Now, on the strength of Theorem \ref{dimsmooth}, we can
conclude from the quasi-stability estimate above that $\mathcal{A}$ has a finite fractal dimension.
\par
Additionally, Theorem \ref{dimsmooth} guarantees that $$||u_{tt}(t)||^2+||u_t(t)||_2^2 \le C(\mathcal A)~\text{ for all } t \in \R.$$
Since $u_{t}\in H^{2}(\Omega )$, and via the sharp regularity of the Airy stress function $$||f_V(u)|| \le C||u||_2^2+||u||_2||v(u)||_2 \le C\Big(||u||_2^2+||u||_2^3\Big) \le C(\mathcal A),$$ standard elliptic regularity theory for
$$
\Delta ^{2}u=-u_{tt}-k u_t-f_V(u)-Lu+p \in L^2(\Omega),
$$
with the clamped boundary conditions give that $$\ ||u(t)||_{4}^{2}\leq C(\mathcal A) ~~~\text{for all}~ ~
t\in \reals.$$ 
Thus, we can conclude  additional regularity of the trajectories from the attractor $$\mathcal A \subset W=(H^4\cap H_0^2)(\Omega) \times H_0^2(\Omega)$$ (with associated bound), as stated in Theorem \ref{th:main1}.

\subsection{Exponential Attractor with Large Damping}\label{strong}
In obtaining the quasi-stability estimate on the attractor $\mathcal A$ we critically utilized the compactness of the attractor. In order to use Theorem \ref{expattract*} (which yields a generalized fractal exponential attractor), we show the quasi-stability estimate on an absorbing set $\mathcal B \subset Y$, and verify a H\"{o}lder continuity property of the dynamics in a weaker topology. 

\begin{proof}[Proof of Theorem \ref{largek1}] On any bounded set $B$ (for instance on the absorbing set constructed in the proof of Theorem \ref{disp}), one can show the quasi-stability estimate by simply taking $k$ large---this follows directly from \eqref{enest1} and \eqref{stufff} with $k$ sufficiently large. 
\begin{remark} Note that the proof carried in Lemma \ref{est*} does not apply here, as we do not have a compact set upon which to work. \end{remark}

H\"{o}lder continuity in the space $\widetilde Y\equiv [L^2(\Omega) \times H^{-2}(\Omega)]$ is obtained as follows\footnote{We can first consider strong solutions, and extend by density.}: Consider a function $\zeta \in H_0^2(\Omega)$ and multiply \eqref{difference} by $\zeta$, integrating over $\Omega$:
\begin{equation}
(z_{tt},\zeta)_{\Omega} \le ||z||_2~||\zeta||_2+k||z_t||~||\zeta||+||\cF(z)||_{-\eta}~||\zeta||_{\eta}+C||z||_{2-\eta}~||\zeta||
\end{equation}
Integrating from $s$ to $t$, we have:
\begin{equation}
\big|(z_t(t)-z_t(s),\zeta)\big| \le C(R)||\zeta||_2\int_s^t\Big(||z||_2+||z_t||\big)d\tau \le C(R)||\zeta||_2|t-s|,
\end{equation}
where we have appealed to the locally Lipschitz nature of the nonlinearity $f_V: H_0^2(\Omega) \to L^2(\Omega)$  owing to the sharp Airy stress function regularity and the existence of the absorbing set $\mathcal B$.

In addition, we note
\begin{equation}
||z(t)-z(s)|| = \left|\left|\int_s^t z_t(\tau) d\tau\right|\right| \le C(R)|t-s|.
\end{equation}
H\"{o}lder continuity (in the sense of \eqref{holder}) of the dynamics in  the topology of $\widetilde Y$ follows. Thus, we obtain Theorem \ref{largek1} from Theorem \ref{expattract*}. 
\end{proof}
\begin{remark}
We note here that Theorem \ref{largek1} was established in \cite{springer} as described above. We discuss it here primarily for contrast with what is presented in the sequel. 
One of the features of this approach is that the obtained exponential attractor---though a subset of the state space element-wise---may be unbounded in the corresponding  topology of $Y$.\end{remark}

\section{Analysis of the Berger Plate: $f=f_B$}
In this section we focus on $f=f_B$. Certain aspects of the analysis are simplified from the case of $f=f_V$. Additionally, if we consider the damping coefficient to be large (as in the hypotheses of Theorem \ref{largek1}), we obtain even stronger results than Theorem \ref{largek1} above.
\subsection{Finite Dimensional, Smooth Global Attractor AND Generalized Fractal Exponential Attractor in One Shot}
Since the structure of $f_B$ is simplified from that of $f_V$, we can show the quasi-stability estimate on any bounded, forward invariant set. 
For this, we demonstrate a decomposition of
Berger nonlinearity for the difference of two solutions. From this lemma the completion of the proof of Theorem \ref{th:main1} and Theorem \ref{th:main2} follow easily as corollaries. 
\begin{lemma}\label{decomp}
Let $z=u^1-u^2$, and let $f_B(u)=(b -||\nabla u||^2)\Delta u$,
and $\mathcal F(z) = f_B(u)-f_B(w)$. Also assume $u^1,u^2 \in
C(s,t;(H^2\cap H_0^1)(\Omega))\cap C^1(s,t;L^2(\Omega))$. Then
\begin{align*}
\int_s^t (\mathcal F(z),z_t)_{\Omega} d\tau  =&-\dfrac{1}{2}\Big[\big(b-||\nabla u^1||^2\big)||\nabla z||^2\Big]_s^t \\
&~-\int_s^t\big(||\nabla u^1||^2-||\nabla u^2||^2\big)(\Delta u^2,z_t) d\tau \\
&~-\int_s^t(\Delta u^1, u^1_t)||\nabla z||^2d\tau
\end{align*}
\end{lemma}
\begin{proof}[Proof of Lemma \ref{decomp}]
Letting $z=u^1-u^2$, and letting $B(u)=(b-||\nabla u||^2)$, we note:
\begin{align*}
( \mathcal F(z),z_t )_{\Omega} =&~\big(B(u^1)\Delta u^1-B(u^2)\Delta u^2,z_t\big)\\
=&~\Big(B(u^1)\Delta z+\big(B(u^1)-B(u^2)\big)\Delta u^2,z_t\Big)\\
=&~B(u^1)\big(\Delta z,z_t)_{\Omega}+\big(B(u^1)-B(u^2)\big)(\Delta u^2,z_t)_{\Omega}\\
=&~-\dfrac{1}{2}B(u^1)\dfrac{d}{dt}||\nabla z||^2+\big(B(u^1)-B(u^2)\big)(\Delta u^2,z_t)_{\Omega}.
\end{align*}
This implies the result. 
\end{proof}

\noindent At this point, restricting to a bounded, forward-invariant set (radius denoted by $R$)\begin{equation*}
||u^1(t)||_{2}+||u^1_{t}(t)||_{0}+||u^2(t)||_{2}+||u^2_{t}(t)||_{0}\leq
C(R),~~t>0,
\end{equation*}
using triangle inequality
\begin{align*}\left|~||\nabla u^1||^2-||\nabla u^2||^2~\right|=&~\Big|||\nabla u^1||-||\nabla u^2||\Big|\left(||\nabla u^1||+||\nabla u^2||\right)\\
\le &~||\nabla u^1 - \nabla u^2||\left(||\nabla u^1||+||\nabla u^2||\right)\\
\le &~C(R)||z||_1,
\end{align*}
and taking into account the last two inequalities in Lemma
\ref{decomp} we obtain the {\em key inequality}:
\begin{equation}\label{usethis}\left|\int_s^t(\mathcal F(z),z_t)d\tau \right| \le \epsilon\int_s^tE_z(\tau)d\tau+C(\epsilon,R,T)\sup_{[s,t]} ||z||^2_{2-\eta},~~\eta>0.\end{equation}

As before, on any bounded set $B$ (for instance on the absorbing set constructed in the proof of Theorem \ref{disp}), we obtain the quasi-stability estimate with \eqref{usethis} in \eqref{enest1}. Hence, for $(S_t,Y)$ with $f=f_B$, by Theorem \ref{doy} taken in conjunction with Theorem \ref{disp}, there exists a compact global attractor $\mathcal A$. In addition, $\mathcal A$ is a bounded, forward invariant set; by Theorem \ref{dimsmooth}, and repeating the analysis at the end of Section \ref{refone} (done for $f=f_V$), the attractor $\mathcal A \subset W$ (bounded in $W$) and has finite fractal dimension. This completes the proof of Theorem \ref{th:main1}.

Noting that (as in the case of $f=f_V$) the dynamics are uniformly H\"{o}lder continuous on the weaker space $\widetilde Y$, we can appeal to  Theorem \ref{expattract*}. Indeed, we have the quasi-stability estimate on $\mathcal B$, and thus the existence of a generalized fractal exponential attractor $\mathcal A_{\exp}$ (with dimension finite in $\widetilde Y$) is guaranteed. This now completes the proof of Theorem \ref{th:main2}.
\begin{remark} We emphasize that in the case of $f=f_B$, owing to \eqref{usethis}, it was not necessary to utilize large $k$ to obtain the quasi-stability estimate on the absorbing ball $\mathcal B$. This is due to the availability of estimate \eqref{usethis}, which was not readily available in the case $f=f_V$. And, as seen above, to obtain the quasi-stability estimate on $\mathcal B$ in that case, it was necessary to take $k$ large.
\end{remark}

\subsection{Proper Exponential Attractor with Large Damping}

In this section we proceed to show Theorem \ref{largek2}; namely, that with $k$ large, the generalized fractal exponential attractor from the previous section (obtained for any damping coefficient $k>0$) can be improved. Indeed, here we will obtain---for sufficiently large damping---a proper fractal exponential attractor whose dimension is finite in the state space $Y$ itself (rather than only in $\widetilde Y$).

The primary theorem utilized in this section is Theorem \ref{exp2}, and is given in \cite[Theorem 3.4.8, p. 133]{quasi}.
The approach utilized here is:
\begin{enumerate}
\item We first show that, by considering large damping---$k(\mathcal B)$---there is a smooth set $\mathcal S$ (i.e., contained in $W\cap \mathcal B$) which is uniformly exponentially attracting. To this end, we will decompose the nonlinear dynamics into a smooth dynamics and an exponentially stable dynamics in the topology of the state space. 

\item Having {\em constructed} $\mathcal S$ we can operate directly on this set; since we have that the quasi-stability estimate holds on $\mathcal B$ (or any forward invariant set), it will certainly apply on $\mathcal S$ and we shall apply Theorem \ref{exp2}.
\end{enumerate}

\subsubsection{Decomposition of Finite-Energy Dynamics}\label{decomp*}
Denote by $\mathcal B \subset Y$ a fixed absorbing ball of radius $R$ for the dynamics $(S_t,Y)$ taken with $f=f_B$. We have $\mathcal A \subset \mathcal B$. We note that, as mentioned above, for $k>k_*>0$ the size of the absorbing ball $\mathcal B$ can be taken uniform in $k$.  We note that by waiting sufficiently long, we can reduce dynamics corresponding to any initial data to dynamics with initial data chosen from $\mathcal B$. We will denote this dependence by $R$---the radius of $\mathcal B$---and consider an initial time $t_0$ such that $(u(t_0),u_t(t_0)) \in \mathcal B$. 

In this section we prove the following Theorem:
\begin{theorem}\label{stuff}
Let $R$ denote the radius of the absorbing set $\mathcal B$ in $Y$, and recall the notation $W\equiv (H^4\cap H_0^2(\Omega)) \times H_0^2(\Omega)$.
Then ~$\exists~M(R), C(R),$ and $\gamma>0$ such that for any bounded $B\subseteq Y$ 
\begin{equation}\label{yep} d_Y\Big(S_t(B), \overline{B_M\left( W\right)} \Big) \le C_R e^{-\gamma t},\end{equation}
where $B_M(W)$ is a ball of radius $M$ in the space $W$. By construction $\overline{B_M(W)} \subset \mathcal B$.
\end{theorem}

\begin{proof}[Proof of Theorem \ref{stuff}] We will  apply the following decomposition of  the solution\footnote{The decomposition is inspired by, but different in structure, than the decomposition that has been used for the wave equation
in the proof of exponential attraction i.e.,  \cite{fgmz}.}: 

Take $(u(t_0),u_t(t_0)) \in \mathcal B \subset Y$. We write $u = z+w$, where
 $z$ and $w$ correspond to the systems:
\begin{equation}\label{exp}\begin{cases} z_{tt} + k z_t + \beta z+ \Delta^2 z +\cF(z,u)  = Lz~~\text{ in } \Omega\\[.2cm]
\cF(z,u) =  [b-||\nabla u||^2]\Delta z\\[.2cm]
z = \Dn z =0 ~~\text{ on } \Gamma,\\
z(t_0) = u(t_0),~~ z_t(t_0) = u_t(t_0).\end{cases}\end{equation}
\begin{equation}\label{smooth}\begin{cases} w_{tt} + k w_t + \Delta^2 w  = p+ Lw-\cF(w,u)+\beta z~~ \text{ in } \Omega\\[.2cm]
 \cF(w,u) = [b-||\nabla u||^2]\Delta w\\[.2cm]
w=\Dn w=0~~\text{ on } \Gamma,\\ w(t_0) = 0, ~~w_t(t_0) = 0,\end{cases}\end{equation}
where the constant $\beta > 0$ will be selected suitably large. 
\begin{remark}
Since we are taking null initial data for the the $w$ portion of the decomposition, we note that the presence of $u$ in the term $\mathcal F(w,u)$ is the ``driver" of the dynamics. 
\end{remark}
 
Given $u$, the $z/w$ system is well-posed with respect to finite energy topology $Y$. Indeed, in light of the stability estimates shown below for the $z$ dynamics (with $u$ {\em given}) the general analysis for abstract plate equations in \cite{springer} applies; from there, the $w$ dynamics are well-posed as a difference $w=u-z$. We state this as a lemma:
\begin{lemma}
Consider the dynamics in \eqref{exp} at time $t_0$, with initial data \\ $(u(t_0),u_t(t_0)) \in \mathcal B \subset Y$. Given $u \in C([t_0,T];H_0^2(\Omega))$ the dynamics corresponding to \eqref{exp} are well-posed in the sense of generalized solutions \cite{springer}. Then, considering $w=u-z$, with  $u$ corresponding to the solution \eqref{plate} on $Y$, the problem \eqref{smooth} is well posed on $Y$ with $(w,w_t)$ as a solution.
\end{lemma}
\noindent Let $$E_{\beta}(z(t))\equiv \dfrac{1}{2}\big[||\Delta z||^2+||z_t||^2+\beta ||z||^2\big].$$ For the decomposed plate dynamics, we will show the following supporting lemmas, which together yield Theorem \ref{stuff}:
 \begin{lemma}\label{expattract}
There exists  $k_e(R)>0$ and  $\beta_e(R)>0$ such that for all $k>k_e$ and all $\beta>\beta_e$, the quantity $E_{\beta}(z(t))$ decays exponentially to zero. In fact, $S^z_t(y_0)=(z,z_t)$ decays uniformly exponentially to zero in $Y$ for $y_0 \in B_{Y}(R)$ (with rate depending on $k$ and $\beta$).
 \end{lemma}

  \begin{lemma}\label{regattract} There exists a $k_Q(R)$ such that for any $k>k_Q$ all $\beta >0$ the evolution $(w,w_t)$ on $Y$ corresponding to \eqref{smooth} has that  $(w,w_t,w_{tt}) \in  C([t_0,\infty);W\times L^2(\Omega)).$ 
 \end{lemma}
 
 \begin{proof}[Proof of Lemma \ref{expattract}]
We consider the Lyapunov approach, as was used above in Section \ref{dissp} (also used in \cite{springer,delay}, based upon \cite{Memoires}) to show the existence of an absorbing set. Define the Lyapunov-type function:
\begin{align}
\widehat V(S^z_t(y_0)) \equiv &~E_{\beta}(z)+\big( z_t,z\big)_{\Omega} +\frac{k}{2}||z||^2
\end{align}
where $S^z_t(y_0) = (z(t),z_t(t))$ for $t \ge t_0,$ and $\beta$ is some positive number to be specified later. Using Young's inequality, we have immediately that
\begin{equation}\label{energybounds*}
cE_{\beta}(z(t)) \le \widehat V(S^z_t(y_0)) \le C(k)E_{\beta}(z(t))
\end{equation}
The constant $c$ does not depend on either damping parameter.
We consider the analog of the analysis in Section \ref{dissp}; rather then employing Proposition \ref{potentiallowerbound} on lower order terms (which yields a constant---see Lemma \ref{le:48}), we may scale $\beta$ instead. This yields:
\begin{lemma}\label{le:48*}
There exists $k_e$ and $\beta_e$ so that for $k>k_e>0$ and $\beta>\beta_e>0$:
\begin{equation}\label{goodneg*}
\dfrac{d}{dt}\widehat V(S^z_t(x))\le-c(k,\beta)\Big\{||z_t||^2+||\Delta z||^2 +||z||^2\Big\},
\end{equation}
where $c(k,\beta) \to \infty$ as $\min\{k,\beta\}\to \infty$.
\end{lemma}
From this lemma, and the upper bound in \eqref{energybounds*}, we have for some $\gamma(k,\beta)>0$: 
\begin{equation*}
E_{\beta}(z(t))\le C\widehat V(S^z_t(y_0)) \le C \widehat V(y_0)e^{-\gamma t},~~t \ge t_0.
\end{equation*} 
This concludes the proof of Theorem \ref{expattract}, the uniform exponential stability of the $z$ portion of the decomposed dynamics on $\mathcal B$.
\end{proof}

\begin{proof}[Proof of Lemma \ref{regattract}]
Again, recall that we have chosen $t_0$ large enough so that $(u(t_0),u_t(t_0)) \in \mathcal B$. Let $S^w_t(\mathbf 0)=(w(t),w_t(t))$ correspond to the \eqref{smooth} dynamics.
First, we note: 
\begin{lemma}\label{heyo}
The dynamics $S_t^w(\mathbf 0)$ (such that $S_t^w(\mathbf 0)+S_t^z(y_0)=S_t(y_0)$, and $u=w+z$) are uniformly bounded on $Y$. This is to say that:
$$||(w(t),w_t(t))||_{Y} \le C(R), ~\forall~~t\ge t_0.$$\end{lemma}
\begin{proof}[Proof of Lemma \ref{heyo}]
This follows immediately from the existence of a uniform absorbing set $\mathcal B$ for the dynamics $S_t(y_0) = (u(t),u_t(t))$ on $Y$, and the uniform exponential stability of the $S^z_t(y_0)$ dynamics. \end{proof}

To continue with the proof of Lemma \eqref{regattract}, recall $\cF(w,u) = [b-||\nabla u||^2]\Delta w$. Consider the time-differentiated $w$ dynamics; let $\overline w = w_t$:
\begin{equation}\label{whatthe}\begin{cases} \overline w_{tt} + k\overline w_t + \Delta^2 \overline w  = -\dfrac{d}{dt}\left\{\cF(w,u)\right\}+L\overline w+\beta z_t~~ \text{ in } \Omega\\
\overline w=\Dn \overline w=0~~\text{ on } \Gamma\\
\overline w(t_0) = 0, ~~\overline w_t(t_0) = 0.\end{cases}\end{equation}

Recall, the term~ $||z_t||$ decays exponentially, and is thus bounded. We will now obtain a global-in-time-bound on~~$||\overline w||_2+||\overline w_t||$~ in what follows. We now invoke the well known {\em exponential decay of the linear, damped (static {and} viscous) plate equation}\footnote{A proof of this decay is essentially given in the argument above utilizing \eqref{energybounds*}--\eqref{goodneg*}.}  for 
$$\overline w_{tt} + k\overline w_t + K \overline w + \Delta^2 \overline w =G.$$ It is critical to note that the margin of stability is controlled by $k,K$:  ~$\omega(k,K) \to \infty$ ~as ~$\min [k,K] \to \infty$~ (see the analysis yielding \eqref{goodneg} and \eqref{goodneg1}). We consider:
\begin{align} 
\overline w_{tt} + k \overline w_t+K\overline w + \Delta^2 \overline w=&~  -\dfrac{d}{dt}\left\{\cF(w,u)\right\}+L\overline w+\beta z_t+K\overline w\\\equiv&~  G(u,z,w,\overline w, w^t).
\end{align}  Let $\mathbf w(t) = (\overline w(t),\overline w_t(t))$. For $t_0$ (again, sufficiently large) we utilize the {\em variation of parameters formula}:
\begin{multline}\label{varpar0}
\Big|\Big|\mathbf w(t)\Big|\Big|_{Y}  \le 
C \Big|\Big|\mathbf w(t_0)\big)\Big|\Big|_{Y}\\+\int_{t_0}^t e^{-\omega(k,K)(t-s)}\Big[\left|\left|\dfrac{d}{dt}\cF(w,u)\right|\right|_0+||L\overline w||_0+K||\overline w||_0+\beta||z_t||_0\Big]ds.
\end{multline}
 The critical term is $\dfrac{d}{dt}\mathcal F(w,u)$, as the other terms are bounded by Lemma \ref{heyo}. We have:
\begin{align}\label{timenon}
\Big|\Big|\dfrac{d}{dt}\cF(w,u) \Big|\Big|_0
= \Big|\Big|(\nabla u, \nabla u_t)\Delta w+||\nabla u||^2\Delta w_t \Big|\Big|_0
\le~\left|\left|(\Delta u,u_t)\Delta w\right|\right|+||\nabla u||^2\left|\left|\Delta \overline w\right|\right|.
\end{align}
\begin{remark}
This is the key step for which no analog exists if performing this decomposition analysis on the von Karman plate ($f=f_V$). 
\end{remark}
Thus: \begin{equation}\label{keyvK} \left|\left| \dfrac{d}{dt}\cF(w,u)\right| \right|_0 \le C(R)\big(1+||\Delta \overline w||\big).\end{equation}
 Returning to \eqref{varpar0}, and implementing the above bounds in \eqref{keyvK}, we have
\begin{align}
\Big|\Big|\big(\overline w(t),\overline w_t(t)\big)\Big|\Big|_{Y} \le \nonumber&~
C ||(\overline w(t_0),\overline w_t(t_0))||_{Y}\\&+C(R)\int_{t_0}^t e^{-\omega(k,K)(t-s)}\Big[1+K+||\overline w||_2\Big] ds 
\end{align}
 Thus:
\begin{align}
\Big|\Big|\big(\overline w(t),\overline w_t(t)\big)\Big|\Big|_{Y}
\le &~
C ||(\overline w(t_0),\overline w_t(t_0))||_{Y}\\\nonumber&+C(R)\big[1+K+\sup_{s \in [t_0,t]}||(\overline w,\overline w_t)||_{Y}\big]\int_{t_0}^t e^{-\omega(k,K)(t-s)}ds \\
\le &~C\Big|\Big|\big(\overline w(t_0),\overline w_t(t_0)\big)\Big|\Big|_{Y}\\\nonumber&+C(R)\big[1+K+\sup_{s \in [t_0,t]}||(\overline w,\overline w_t)||_{Y}\big]\dfrac{1-e^{-\omega(k,K)(t-t_0)}}{\omega(k,K)}.
\end{align}
A global-in-time bound on $(w_t,w_{tt})=(\overline w, \overline w_t)$ in $Y$ will now follow: (i) we first take suprema in $t \in [t_0,\infty)$. Then (ii) we may choose ~$\min[k,K]$ ~sufficiently large (and thus up-scale $\omega(k,K)$) such that
$$\dfrac{C(R)}{\omega(k,K)}\Big[1-\exp\big(-\omega(k,K)(t-t_0)\big)\Big]<1.$$  Absorbing this term in the LHS completes the bound on $(\overline w, \overline w_t)$.
At this point, elliptic theory can be applied: a bound on $\|w\|_4$ is to be calculated from an elliptic equation in terms of ~$\|(\overline w,\overline w_t) \|_{{Y}}$.
Indeed, we consider biharmonic problem with  the clamped boundary conditions.
$$ \Delta^2 w = - w_{tt} - kw_t  -\mathcal F(w,u) +Lw +\beta z;~~~~w = \Dn w =0 ~\text{on}~ \partial \Omega.$$
This gives, via elliptic estimates (as in Section \ref{smuth}) 
\begin{equation}\label{w4}
\|w(t) \|_4 \leq  C(R)\left[\|(\overline w, \overline w_{t})\|_{Y} +  1\right].
\end{equation}
This completes the proof of Lemma \ref{regattract}. 
\begin{remark}
We note that, above, $K$ appears on the RHS of the estimate and, thus, the global-in-time bounds for the $(\overline w,\overline w_t)$ dynamics depend on the size of $K$, an auxiliary quantity which itself depends on the intrinsic parameters in the problem in the above argument.
\end{remark}
\end{proof}
Taking Lemmas \ref{expattract} and \ref{regattract} together (as in the proofs of \cite[Proposition 5.1]{fgmz} and \cite[Theorem 5.1]{kalzel}) concludes the proof of Theorem \ref{stuff}.

\end{proof}

\subsubsection{Application of Theorem \ref{exp2}}
Here we operate on the smooth set $\mathcal S =\overline{B_K(W)}$ constructed in Theorem \ref{stuff} and apply Theorem \ref{exp2}.
\begin{proof}
Recall $\mathcal S \subset \mathcal B\cap W$; it is both positively invariant (by construction) and compact in $Y$ (as a subset of $W$). This addresses Hypothesis 1 in Theorem \ref{exp2}. Additionally, the dynamical systems requirement (Hypothesis 2) is satisfied by \eqref{dynsys} in Proposition \ref{p:well}. Since we have shown---for $f=f_B$---that $(S_t,Y)$ is quasi-stable on {\em any} bounded, forward invariant set, this includes $\mathcal S$ as a subset of $\mathcal B$. Hence the quasi-stability property (Hypothesis 3) is satisfied on $\mathcal S$. Finally, we address H\"{o}lder continuity on $\mathcal S$ in the topology of $Y$.

By construction we note that $||\mathcal S||_W \le M$. In particular, for any $(w(\cdot),w_t(\cdot)),~~\forall~~t$ sufficiently large, chosen from $\mathcal S$\footnote{Such $w$ can be considered as elements of $\mathcal B$ with $u=z+w=0+w$.} we have \begin{align}
||w(t)||_4 \le C(M),~ &~~||w_t||_2 \le C(M).
\end{align}
And thus we have from 
\begin{align} w_{tt}=&~-\Delta^2w-kw_t-f_B(w)+p+Lw \\
w_t(t)-w_t(s) =&~\int_s^tw_{tt}(\tau)d\tau, \end{align} 
that
\begin{multline}
||w_t(t)-w_t(s)||_0 \le \int_s^t\Big(||w||_4+k||w_t||_0+(b+||\nabla w||^2)||w||_2+C||w||_{2-\eta}+||p||_0\Big) d\tau \\\le C(M,b,p)|t-s|.
\end{multline}
And, as before, we note
\begin{equation}
||w(t)-w(s)||_2 \le \left|\left|\int_s^t w_t(\tau) d\tau \right|\right|_2 \le \big(\sup_{t}||w_t||_2\big)|t-s|\le C(M)|t-s|.
\end{equation}
Thus we note that $(S_t,\mathcal S)$ is uniformly H\"{o}lder continuous in the topology of $Y$ (in the sense of Hypothesis 4 of Theorem \ref{exp2}).

On the strength of Theorem \ref{exp2}, this concludes the proof of Theorem \ref{largek2}.
\end{proof}

\section{Numerical Simulations}\label{numerics}

\def\Ucrit{U_{\text{crit}}}
\def\hUcrit{\hat{U}_{\text{crit}}}
We now demonstrate several qualitative aspects of the piston-theoretic dynamics considered herein. To do so we conduct numerical simulations on the 1-D (beam) analogue of the specific, piston-theoretic plate in \eqref{plate-stand}
\begin{equation}\label{beam-stand}
u_{tt}+\partial_x^4u+ku_t+\lambda f_B(u) = p-\mu Uu_x ~~ \text { in } ~[0,\ell]\times (0,T)
\end{equation}
taken with clamped-clamped boundary conditions. With respect to damping, we take $k=k_0+\mu$ as the damping parameter;  $\mu$ represents the scaling of the frictional damping due to the piston-theoretic term (that is, due to flow effects), and $k_0$ represents the material/imposed damping. We introduce $\lambda$ as a binary flag that indicates whether or not the Berger nonlinearity is in force:
\begin{equation}
\label{bergerdyn}
f_B(u)=\left(b-b_0\|u_x\|^2_{L^2([0,\ell])}\right)u_{xx}.
\end{equation}
Here $b$ represents beam tension $b>0$ or compression $b<0$ (at equilibrium); the parameter $b_0$ represents the strength of the nonlinear effect of stretching on bending. Computations are focused on both the linear model ($\lambda= 0$ in \eqref{beam-stand}), and the Berger dynamics ($\lambda=1$), with corresponding energies (see Section \ref{en})
\begin{equation}
E(t)=\frac{1}{2}\left[\|u_{xx}(t)\|^2+\|u_t(t)\|^2\right],
\label{linen}
\end{equation}
\begin{equation}
\mathcal{E}(t)=E(t)+\Pi_B(t) = E(t)+\frac{b_0}{4}\|u_x(t)\|^4-\frac{b}{2}\|u_x(t)\|^2-\int_0^\ell p u(t)\, dx.
\label{nonlinen}
\end{equation}
Computations were conducted using second-order accurate finite difference approximations for spatial derivatives\footnote{The accuracy of which was verified via forcing a known true solution.}, together with the MATLAB {\tt ode15s} stiff ODE solver for temporal integration on the reduced first-order system.  All simulations were performed with a spatial mesh size of $\Delta x=\ell/100$.

For the simulations below, we take $\mu=1$ and $\ell=1$.\footnote{Of course both $\ell$ and $\mu$ can affect the in/stability of the linear model; these are not the primary coefficients of interest here.} (Typical physical choices of parameters for varying materials and physical applications can be found in several references, for instance \cite{dowell,jfs,vedeneev}.)  The static pressure $p$ is here taken to be zero.
Initial conditions for the simulations below consist of null initial displacement ($u_0(x)=0$) and a parabolic initial velocity profile ($u_1(x)=10x(1-x)$).   For a particular choice of damping parameter $k$, the critical flow velocity represents the value of $U$ at which the linear dynamics become unstable (in the sense of exponential growth in time). 

In this setup, there are three principal regimes of behavior for the linear dynamics $\lambda=0$. In each case, the effects of the initial conditions are conservative (which is to say that for $k=0$ and $U=0$ the energy is constant throughout deflection).
\begin{itemize}
\item~[$k=0$] With no imposed damping, the transient dynamics are not damped out. Additionally, with $U>0$, the non-dissipative piston term $-Uu_x$ on the RHS induces non-trivial behavior in the displacements, as well as oscillatory behavior in the linear energy $E(t)$.  As $U$ is increased the effect of the flow becomes more pronounced and the location of the spectrum is shifted. At a certain value $U_{\text{crit}}$ the dynamics become unstable and exponential growth (in time) of the solution is observed. 
\item~[$k>0$ fixed] With imposed damping in the model the initial transient dynamics are damped out. We can think of $U_{\text{crit}}$ as an increasing function of the damping $k$ in the problem. For $U<U_{\text{crit}}(k)$ energy is dissipated out of the model and solutions are, in fact, (exponentially) stable. 
\item~[$k>0$, $U>U_{\text{crit}}(k)$] Even with damping, for $U>U_{\text{crit}}(k)$, solutions exhibit exponential growth in time. We remark that, for a given unstable $U$, one can consider increasing $k$ to stabilize the dynamics. 
\end{itemize}

\subsection{Influence of Damping on Energy---Without Berger Effects}

First, computed energies $E(t)$ for the linear model ($\lambda=0$) with $k=0$ (no damping) are shown in Figure \ref{fig1}. From a physical point of view, this is an artificial case; we have no imposed damping, and we have ``turned off" the damping due to the flow (on the RHS of \eqref{plate-stand}). For this choice of $k$, an empirically determined approximation to the critical flow velocity is $\Ucrit=636$.  Energy profiles (log-scale) for various choices of $U$ in terms of $\Ucrit$ are given---a linear profile in the energy plot represents exponential growth or decay of the dynamics in the finite energy topology, with margin of in/stability depending on the slope of the profile.  Note the exponential growth in energies for $U\ge\Ucrit$ while $E(t)$ remains bounded for $0<U<\Ucrit$, and is constant for $U=0$.  
\begin{figure}[htp]
\begin{center}
\includegraphics[width=5in]{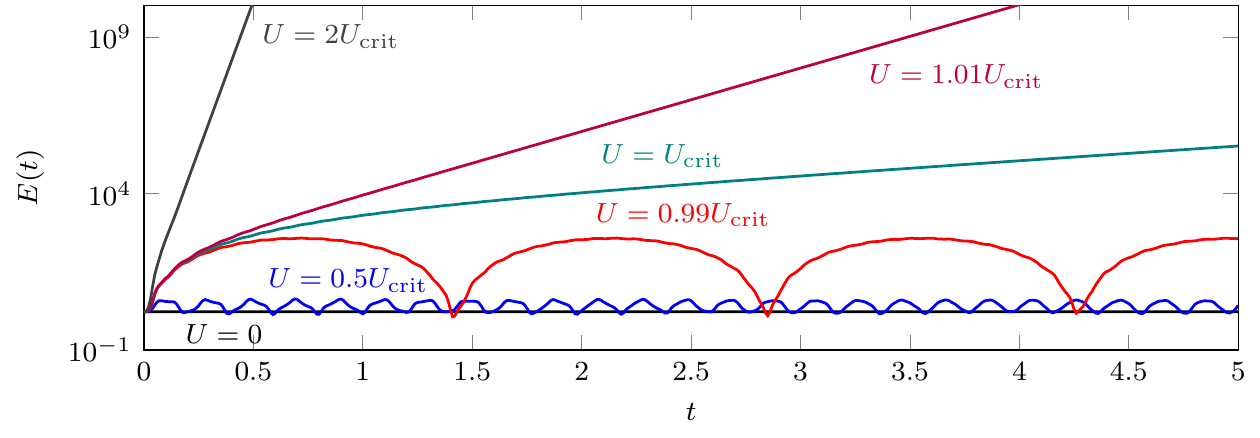}\\
\vspace*{-0.2in}
\caption{Plot of $E(t)$ for $k=0$ and varying $U$, $U_{\text{crit}}=636$.}\label{fig1}
\end{center}
\end{figure}

In the next simulation the damping parameter is set to $k=1$ in the linear model ($\lambda=0$). This can be thought of as including flow damping in the model, without structural (or imposed) frictional damping. This results in an approximate critical flow velocity of $\hUcrit=637$ (the presence of damping perturbs the critical flow velocity corresponding to the onset of instability by roughly $0.2\%$).  Computations were performed again for various choices of $U$ in terms of $\hUcrit$, with energies shown in Figure \ref{fig2}.  Note the effect that the damping has on the decay of energies for all $U<\hUcrit$ (as $E(t)\to 0$ when $t\to\infty$)---each curve for $U<\hUcrit$ has a linear profile (or envelope) with negative slope, indicating exponential decay.  Energy values for $U\ge \hUcrit$ still grow exponentially with damping in this linear model.
\begin{figure}[htp]
\begin{center}
\includegraphics[width=5in]{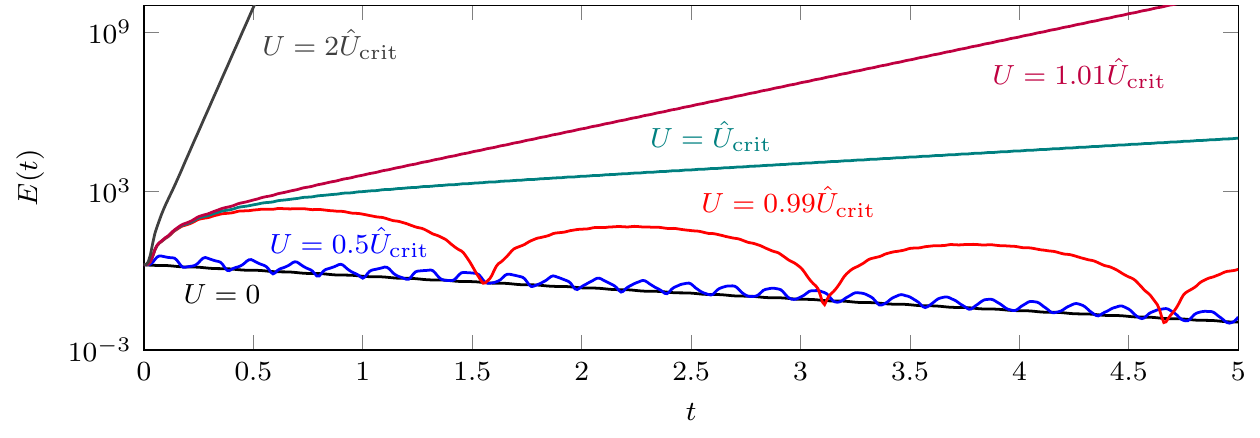}\\
\vspace*{-0.2in}
\caption{Plot of $E(t)$ for $k=1$ and varying $U$, $\hat U_{\text{crit}}=637$, linear model.}
\label{fig2}
\end{center}
\end{figure}

To demonstrate the effect of increasing the damping parameter $k$ for a fixed flow velocity $U$, Figure \ref{fig3} gives the midpoint displacement of the beam for $U=600$ and varying choices of $k$.  At this $U$, for $k=0$ a wave packet-like behavior of the displacement is observed.  As $k$ increases the oscillations are damped, more quickly for higher $k$.
\begin{figure}[htp]
\begin{center}
\includegraphics[width=5in]{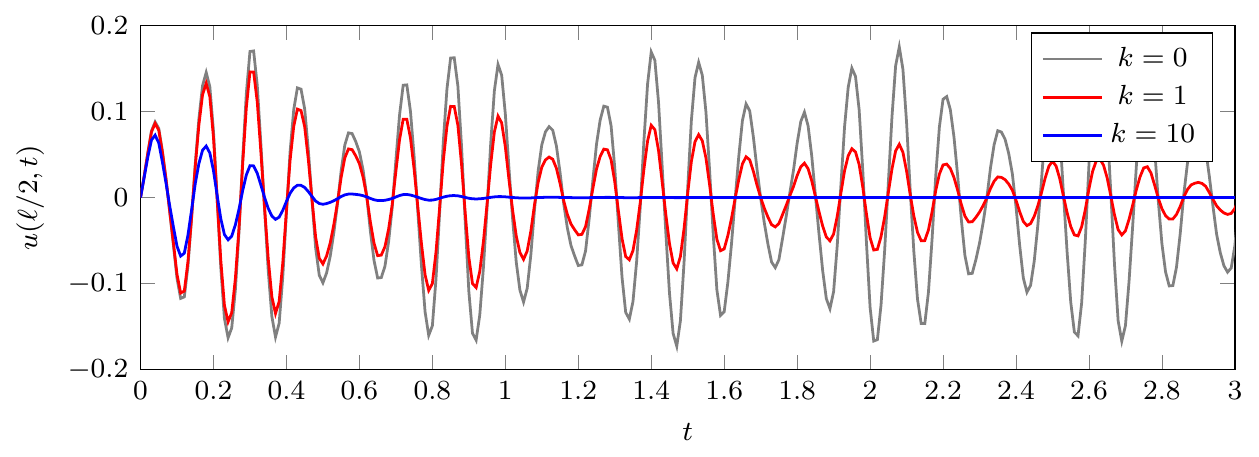}\\
\vspace*{-0.2in}
\caption{Plot of $u$ for $U=600$ and varying $k$, linear model.}
\label{fig3}
\end{center}
\end{figure}

\subsection{Influence of Berger Dynamics on Boundedness of Trajectories}
In Figure \ref{fig4}, the Berger dynamics $f_B(u)$ (with $b=0$ and $b_0=1$) are included into the model and the nonlinear energy $\mathcal{E}(t)$ is shown.  Note that $\mathcal{E}(t)$ (and hence $E(t)$, due to \eqref{enerequiv}) now remains bounded for all supercritical (unstable) velocities $U>\hUcrit$, demonstrating the (Lyapunov) stability that is induced by the $-\|u_x\|^2  u_{xx}$ term. From the point of view of trajectories, each remains bounded for all time, with global-in-time bound dependent upon $U$. Note, for $U>\hat U_{\text{crit}}$, if the nonlinearity is in force ($\lambda=1$) we refer to the non-transient behavior as {\em flutter}, and various qualitative behaviors are possible in this regime. For $U<\hat U_{\text{crit}}$, we note that the nonlinearity does not dramatically affect the stability observed in the linear case.
\begin{figure}[htp]
\begin{center}
\includegraphics[width=5in]{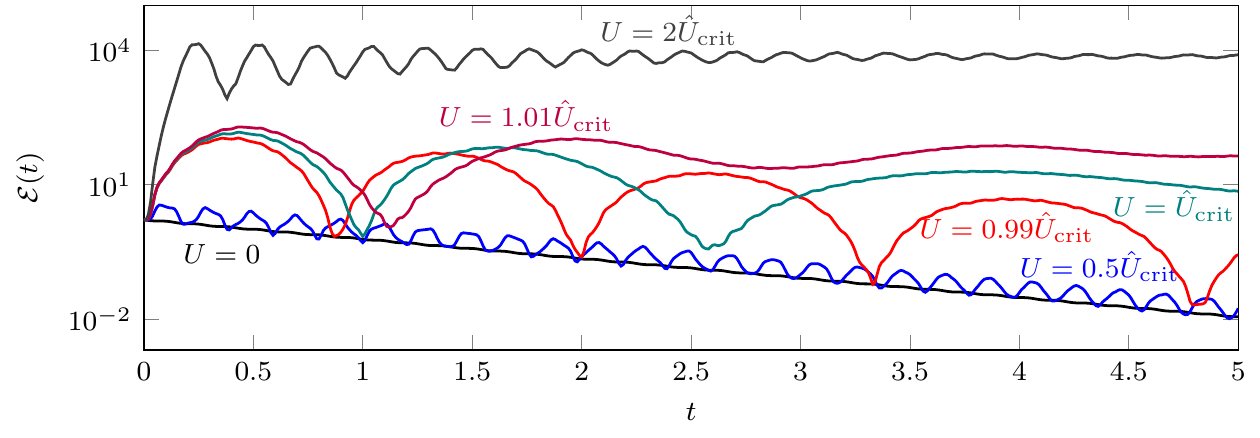}\\
\vspace*{-0.2in}
\caption{Plot of $\mathcal E(t)$ for $k=1$, $b=0$, and $b_0=1$, varying $U$, $\hat U_{\text{crit}}=637$, nonlinear model.}
\label{fig4}
\end{center}
\end{figure}

To show the detailed difference between the energies $E(t)$ and $\mathcal{E}(t)$, Figure \ref{fig5} shows both quantities for $k=1$, $b=0$, and $b_0=1$ with $U=2\hUcrit$.  Figure \ref{fig6} shows the midpoint displacement for the same parameter values. 
\begin{figure}[htp]
\begin{center}
\includegraphics[width=5in]{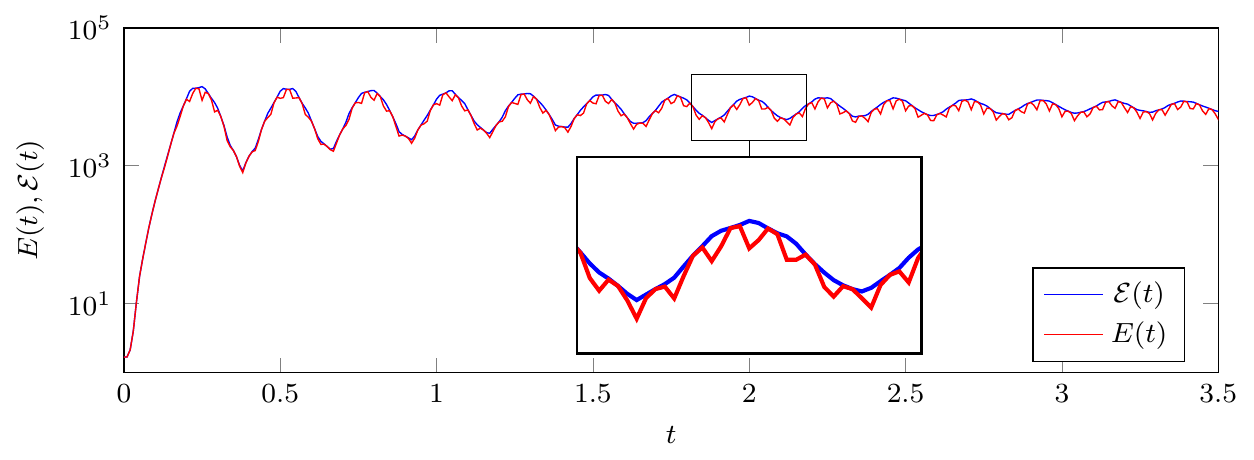}\\
\vspace*{-0.2in}
\caption{Plot of energies for $k=1$ and $U=2\Ucrit$, nonlinear model.}
\label{fig5}
\end{center}
\end{figure}
\begin{figure}[htp]
\begin{center}
\includegraphics[width=5in]{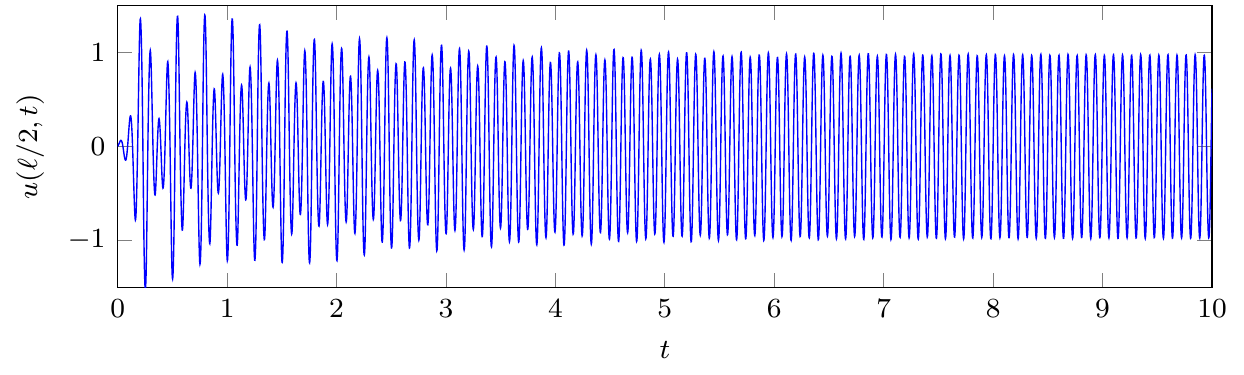}\\
\vspace*{-0.2in}
\caption{Plot of $u$ at beam midpoint for $k = 1$ and $U = 2U_{\text{crit}}$, nonlinear model.}
\label{fig6}
\end{center}
\end{figure}

\subsection{Influence of the $b_0$ Parameter}
The ``restoring" term $-b_0\|u_x\|^2u_{xx}$ in the Berger dynamics maintains the boundedness of trajectories. This can be see through the definition of the nonlinear energy \eqref{nonlinen} and the bound \eqref{enerequiv}. We note that for any $b_0>0$, trajectories will remain bounded (with bound depending on $b_0$).  In simulations for $U=600$, $k=1$, and $b=0$ with several values of $b_0$ we note little qualitative difference in the dynamics for $b_0=0.001$ and $b_0=1$. However for $b_0=100$, we note that both frequency and amplitude of the oscillations are somewhat decreased. 

\subsection{Influence of the $b$ Parameter}

The parameter $b$ in the Berger dynamics translates to in-axis tension if $b<0$ and compression if $b>0$. To demonstrate the effect of this parameter on the beam dynamics, the beam midpoint displacement is plotted in Figure \ref{fig7} for $U=600$, $k=1$, $b_0=0.1$, and several choices of $b$.  Note that in-plane compression has a significant effect on the dynamics by impacting the structure of the set of stationary solutions, while in-plane tension acts as a stabilizing term on the beam oscillations, as well as reduces the frequency of oscillations.
\begin{figure}[htp]
\begin{center}
\includegraphics[width=5in]{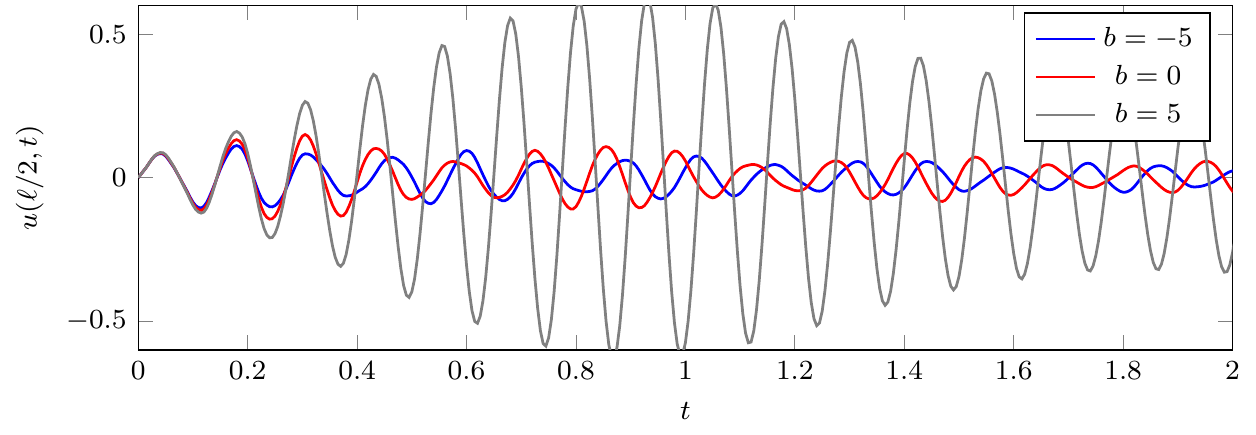}\\
\vspace*{-0.2in}
\caption{Plot of $u$ at beam midpoint for $U=600$, $k=1$, $b_0=0.1$, and varying $b$, nonlinear model.}
\label{fig7}
\end{center}
\end{figure}

\subsection{Convergence to A Non-trivial Steady State}

For certain parameter combinations, the presence of excess in-plane compression leads to a ``buckled'' plate/beam configuration (a bifurcation, for the static problem).  For $b_0=1$, the parameter $b=50$ is large enough for  $U=100$ to impart this behavior.  The transient behavior decays more rapidly to the nontrivial steady state for larger values of $k$.  A plot of the energy $E(t)$ for different values of $k$ is given in Figure \ref{fig8}.  Regardless of the choice of $k$, all simulations for $U=100$, $b_0=1$, and $b=50$ converge to the same nontrivial steady state and same linear energy.  A plot of the nontrivial steady state $u$ (the buckled beam displacement) is given in Figure \ref{fig9}.
\begin{figure}[htp]
\begin{center}
\includegraphics[width=5in]{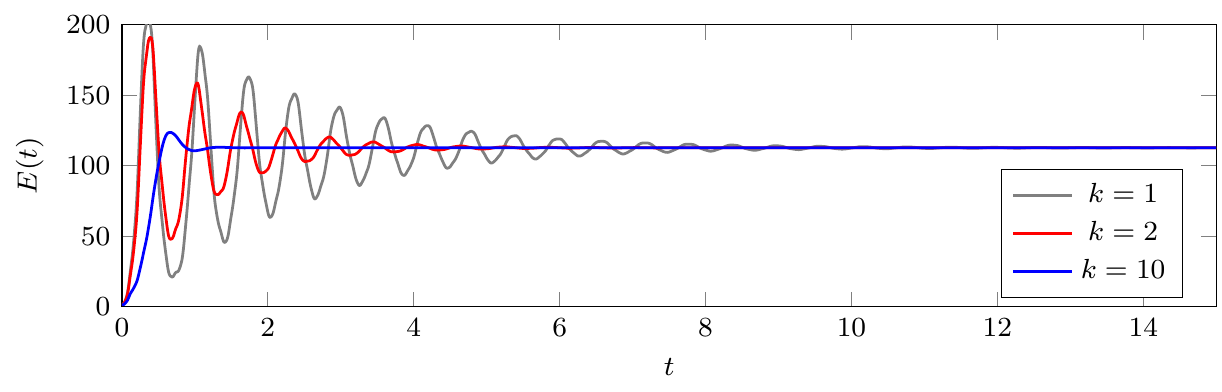}\\
\vspace*{-0.2in}
\caption{Plot of $E(t)$ for for  $U=100$, $b=50$, $b_0=1$, and varying $k$, nonlinear model.}
\label{fig8}
\end{center}
\end{figure}
\begin{figure}[htp]
\begin{center}
\includegraphics[width=5in]{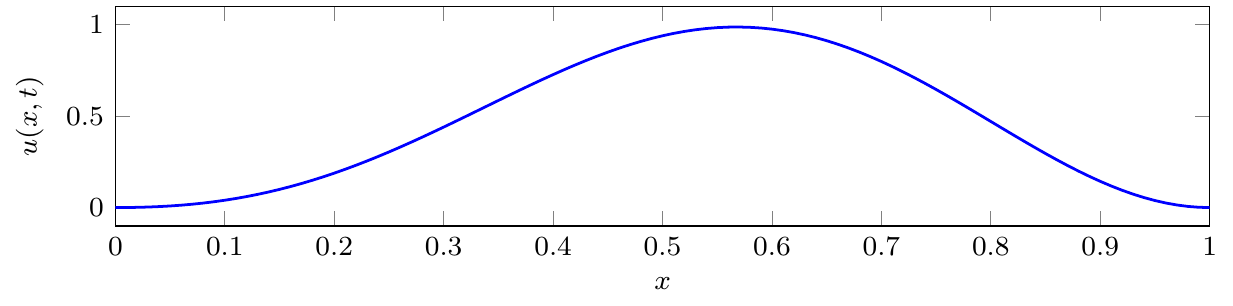}\\
\vspace*{-0.2in}
\caption{Plot of steady-state beam displacement for $U=100$, $b=50$, $b_0=1$, and varying $k$, nonlinear model.}
\label{fig9}
\end{center}
\end{figure}

It is also possible to observe different choices of $k$ resulting in convergence to different steady states.  For $U=100$, $b_0=1$, and $b=100$, the choices $k=1$ and $k=2$ actually produce nontrivial steady states that are negatives of each other.  In Figure \ref{fig10} the midpoint displacement is plotted for these two cases.  
\begin{figure}[htp]
\begin{center}
\includegraphics[width=5in]{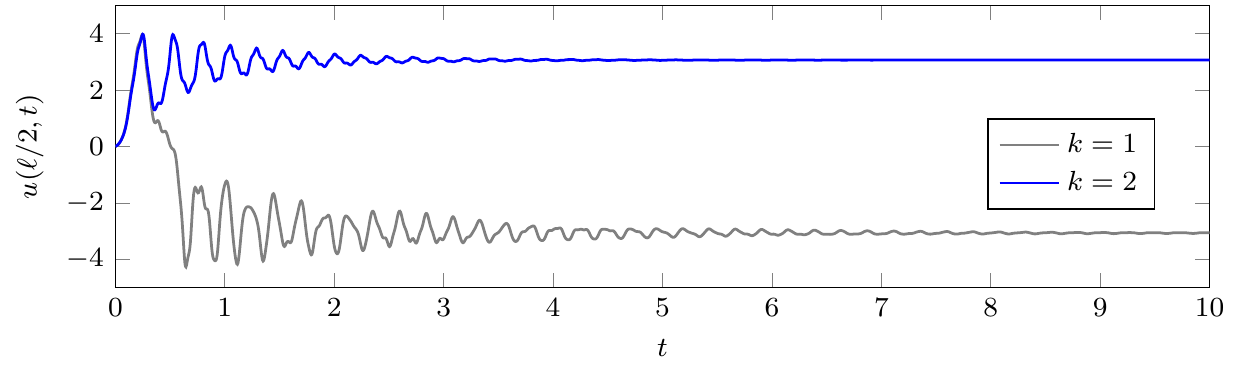}\\
\vspace*{-0.2in}
\caption{Plot of $u$ at beam midpoint for $U=100$, $b=100$, $b_0=1$, and varying $k$, nonlinear model.}
\label{fig10}
\end{center}
\end{figure}

\subsection{Convergence to A Limit Cycle}
In Figure \ref{fig11}, plots of the beam midpoint displacement are shown for a very large flow velocity ($U=5000$---hypersonic), significant in-axis compression parameter ($b=50$), and $b_0=1$ for selected values of the damping parameter $k$.  The sensitivity of the dynamics to the damping parameter can be seen by noting the relative rate of initial decay of energy as $k$ increases---note the quick decay to the oscillatory limit cycle for larger values of $k$.   
\begin{figure}[htp]
\begin{center}
\includegraphics[width=5in]{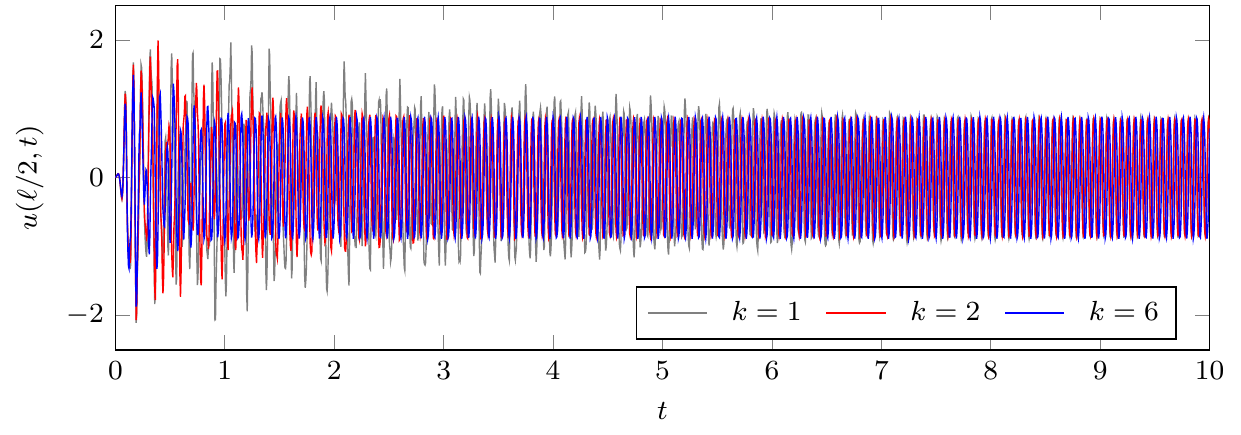}\\
\vspace*{-0.2in}
\caption{Plot of $u$ at beam midpoint for $U=5000$, $b=20$, $b_0=1$, nonlinear model, varying $k$.}
\label{fig11}
\end{center}
\end{figure}

It is possible to induce several different phenomena by manipulating parameter values for the nonlinear dynamics.  In Figure \ref{fig12} the midpoint displacement is shown for $U=5000$, $k=100$, $b=5000$, and $b_0=1000$.  Note the initial transient dynamics are damped out quickly and the dynamics converges to a non-trivial limit cycle. 
\begin{figure}[htp]
\begin{center}
\includegraphics[width=5in]{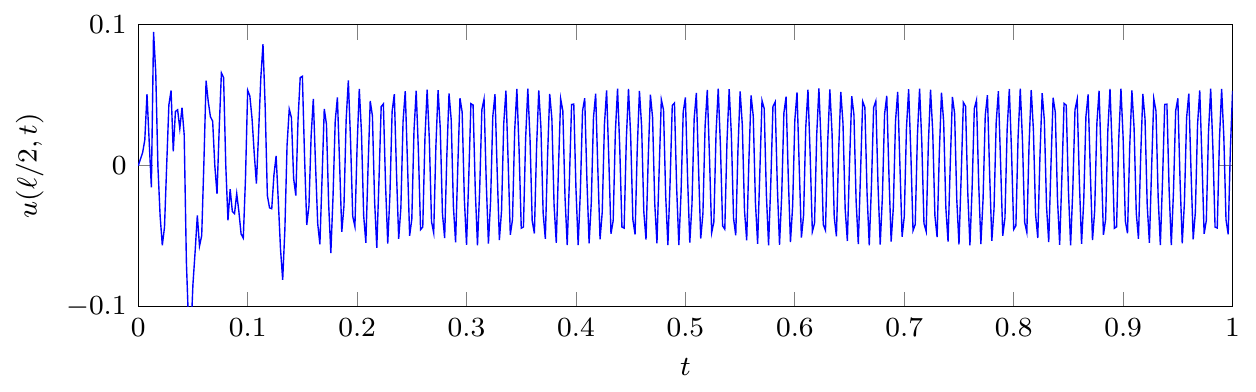}\\
\vspace*{-0.2in}
\caption{Plot of $u$ at beam midpoint for $U=5000$, $b=5000$, $b_0=5000$, $k=101$ nonlinear model.}
\label{fig12}
\end{center}
\end{figure}

\section{Open Questions and Conjectures}
Here we present a brief discussion of open problems and conjectures associated to the analysis above.

We begin with a remark concerning the disparate nature of results shown for $f_B$ and $f_V$. As mentioned in the introduction, $f_B$ is taken to be a simplification of $f_V$, valid for clamped (or hinged) boundary conditions. However, as seen here, there are discrepancies between the nonlinearities (in particular with relation time differentiation of the nonlinearities). On the one hand, as both models are considered valid, it should be the case that results providing qualitative descriptions of the solutions should be the same (or similar). However, it is also reasonable that the simpler Berger model fails to capture some of the nuance and essential dynamics which the von Karman model reflects. 
\subsection{$f=f_B$}
We conjecture that the fractal exponential attractor constructed in the previous section is in fact {\em smooth}. Indeed, the large damping provides a smooth, forward invariant, uniformly exponentially attracting set upon which to work. With this additional smoothness, restricting to this set yields (via the quasi-stability construction \cite{springer} and \cite{quasi}) a fractal exponential attractor with finite dimension in $Y$ (since H\"{o}lder continuity can be shown in this topology). Transitivity of exponential attraction then gives that this restricted fractal exponential attractor is also a fractal exponential attractor for $(S_t,Y)$.  However, the quasi-stability construction of the exponential attractor (see, e.g., \cite[Theorem 7.9.9]{springer} or \cite[Theorem 3.4.7]{quasi} and subsequent remarks) does not address the smoothness properties (or associated bounds) of the exponentially attracting set constructed. Specifically, in this construction (operating on some forward invariant set) we do not have backward dynamics ($t \to -\infty$) to work with.   It is an open question, in general, whether smooth, global attractors are exponentially attracting;  here (with large damping) we have further shown that there is a smooth exponentially attracting set. However, in transitioning between this constructed set and a proper fractal exponential attractor (with finite fractal dimension) {\em smoothness} is potentially lost. 

If piston theory is eschewed, and the full flow-plate system \cite{survey1,survey2,dowellnon,webster} is considered (as well as the reduced, delayed plate corresponding to the ``full potential flow theory"---see \cite{delay,conequil1,conequil2}) the fractal exponential attractor $\mathcal A_{\text{exp}}$ may be shown to be smooth. Indeed, if the reduced plate is obtained from a full flow-plate (gradient) dynamics (see \cite{delay,conequil1,conequil2}) finiteness of the dissipation integral is preserved in the description of the reduced plate. This leads to the following conjecture:
\begin{conjecture}
Consider $f=f_B$ and $k$ sufficiently large (depending on the internal parameters of the problem: $\Omega, L, b,p$). From before, there exists a (proper) fractal exponential attractor ${\mathcal A}_{\text{exp}} \subset Y$, with finite fractal dimension in $Y$, for the dynamics $(S_t,Y)$. Moreover, {\bf if the dissipation integral is finite}: $$\int_{T}^{\infty}||u_t||^2d\tau <\infty,$$ the fractal exponential attractor $\mathcal A_{\text{exp}} \subset W$ (bounded in that topology).
\end{conjecture}
\begin{remark}
In the case of the {\em specific} piston-theoretic flutter dynamics given by \eqref{plate-stand}, one may exploit the clamped boundary conditions to obtain that $$U(u_x,u)_{\Omega}=0.$$ This can be utilized when employing the equipartition multiplier $u$ in estimating $\ds\int_0^T E(t) dt$, and subsequently showing boundedness of the dissipation integral. (This property is not clear when $L$ is a general non-dissipative operator.)
\end{remark}
With the finiteness of the dissipation integral we may consider smooth initial data and the corresponding smooth evolution $S_t$ acting on $W \equiv (H^4\cap H_0^2)(\Omega) \times H_0^2(\Omega)$. This evolution is ultimately dissipative in the topology of $W$ (see \cite[Section 6.2, STEP 1.2]{conequil1}). Moreover, exploiting the structure of $f_B$, we can show that $(S_t,W)$ is in fact quasi-stable on the corresponding absorbing ball for $(S_t,W)$---see \cite[Section 9.5.3]{springer}. In addition, we have immediately that the dynamics $(S_t,W)$ satisfy the uniform H\"{o}lder continuity condition in the topology of $Y$ (as above). By the quasi-stability theory utilized herein (Theorem \ref{expattract*}), this dynamical system $(S_t,W)$ has a generalized fractal exponential attractor, whose dimension is finite in $Y$. Thus a fractal exponential attractor $\overline{\mathcal A}_{\text{exp}} \subset W$ is generated which attracts bounded subsets of $W$ with uniform exponential rate. The transitivity of exponential attraction (Theorem \ref{trans}) guarantees that the exponential attractor $\overline{\mathcal A}_{\text{exp}}$ for $(S_t,W)$ ALSO exponentially attracts bounded sets of $Y$, via the property constructed in \eqref{yep}. This leaves us to conclude that $\overline{\mathcal A}_{\text{exp}} \subset W$ is, indeed, an true exponential attractor of finite fractal dimension in $Y$. Moreover, $\overline{\mathcal A}_{\text{exp}}$ is bounded in the topology $W$. 

\subsection{$f=f_V$}
At present, due to the discrepancy between \eqref{4.9aa} and \eqref{p1} (and thus between \eqref{eq4.5} and \eqref{mmmbop}) it does not seem plausible to obtain the quasi-stability estimate in a vicinity of the attractor $\mathcal A$ without assuming large damping. Additionally, the decomposition approach utilized in proving Theorem \ref{stuff}---which leads to the main result Theorem \ref{largek2}---for $f=f_B$ is not viable for $f=f_V$. In working with the time-differentiated dynamics, the Berger nonlinearity admits the bound \eqref{timenon}; such a bound is not possible for the von Karman dynamics. 

\section{Appendix: Long-Time Behavior of Dynamical Systems}

 Let $(S_t,H)$ be a dynamical system on a complete metric space $H$.    $(S_t,H)$ is said to be ultimately dissipative iff it
possesses a bounded absorbing set $\mathcal B$. This is to say
that for any bounded set $B$, there is a time $t_B$ so that
$S_{t_B}(B)\subset \mathcal B$. 

We say that a dynamical system is
\textit{asymptotically compact} if there exists a compact set $K$
which is uniformly attracting: for any bounded set $ D\subset
H$ we have that $$\displaystyle~
\lim_{t\to+\infty}d_{{H}}\{S_t D|K\}=0$$ 
in the sense of
the Hausdorff semidistance. 
$(S_t,H)$ is said to be
\textit{asymptotically smooth} if for any bounded, forward
invariant $(t>0) $ set $D$ there exists a compact set $K \subset
\overline{D}$ which is uniformly attracting (as above).

A \textit{global attractor} $A\subset H$ is a closed, bounded set in $%
H$ which is invariant (i.e. $S_tA={A}$ for all $%
t>0 $) and uniformly attracting (as defined above). 

\begin{remark} Since we are considering a dynamics which are inherently non-gradient, we do not (here) discuss the set of stationary points, strict Lyapunov functions, or certain characterizations available for dynamical systems which are gradient. \end{remark}

\begin{theorem}\label{dissmooth}
Let $(S_t,H)$ be an ultimately dissipative dynamical system in a complete
metric space  $A$. Then $(S_t,H)$ possesses a compact
global attractor ${A} $ if and only if $(S_t,H)$ is
asymptotically smooth.
\end{theorem}

For non-gradient systems, the above theorem is often the mechanism employed to obtain the existence of a compact global attractor. If one can show that an ultimately dissipative dynamical system $(S_t,H)$ is also asymptotically smooth, one obtains the existence of a compact global attractor. In many cases, showing asymptotic smoothness can be done conveniently using the following criterion due to \cite{kh} and presented in a streamlined way in \cite{quasi,springer}.

\begin{theorem}
\label{psi} Let $(S_t,H)$ be a dynamical system,
$H$ a Banach space with norm $||\cdot||$. Assume that for any
bounded positively invariant set $B \subset H$ and for
all $\epsilon>0$ there exists a $T\equiv T_{\epsilon,B}$ such that
\begin{equation*}
||S_Ty_1 - S_Ty_2||_{\mathcal{H}} \le
\epsilon+\Psi_{\epsilon,B,T}(y_1,y_2),~~y_i \in B
\end{equation*}
with $\Psi$ a functional defined on $B \times B$ depending on
$\epsilon, T,$ and $B$ such that
\begin{equation*}
\liminf_m \liminf_n \Psi_{\epsilon,T,B}(x_m,x_n) = 0
\end{equation*}
for every sequence $\{x_n\}\subset B$. Then $(S_t,H)$ is
an asymptotically smooth dynamical system.
\end{theorem}

A generalized fractal {\em exponential attractor} for the dynamics $(S_t,H)$ is a forward invariant, compact set $ A_{\text{exp}} \subset Y$ in the phase space, with finite fractal dimension (possibly in a weaker topology), that attracts bounded sets with uniform exponential rate.  When we refer to $A_{\text{exp}}$ as a {\em fractal exponential attractor}, we are simply indicating that $A_{\text{exp}}\subset Y$ has fractal dimension in $Y$, rather than in some weaker space. 

 \begin{remark}\label{defmark}
 We include the word ``generalized" to indicate that the finite fractal dimensionality could be shown in a weaker topology than that of the state space.\end{remark}

In Section \ref{strong} we will implicitly appeal to the transitivity of exponential attraction (via Theorem \ref{exp2}); we show it here for the sake of exposition. Loosely, in a fixed topology, the property of a set uniformly exponentially attracting bounded sets in this topology is transitive. We now state this formally---see \cite{fgmz} for discussion and applications to the nonlinear wave equation.
\begin{theorem}[Theorem 5.1, \cite{fgmz}]\label{trans}
Let $(\mathcal M,d)$ be a metric space with $d_{\mathcal M}$ denoting the Hausdorff semi-distance. Let a semigroup $S_t$ act on $\mathcal M$ such that
$$d(S_t m_1, S_tm_2) \le Ce^{Kt}d(m_1,m_2),~~m_i \in \mathcal M.$$
Assume further that there are sets $M_1, M_2, M_3 \subset \mathcal M$ such that
$$d_{\mathcal M}(S_tM_1,M_2) \le C_1e^{-\alpha_1 t},~~d_{\mathcal M}(S_t M_2, M_3) \le C_2 e^{-\alpha_2t}.$$
Then for $C'=CC_1+C_2$ and $\alpha'=\dfrac{\alpha_1\alpha_2}{K+\alpha_1+\alpha_2}$ we have
$$d_{\mathcal M}(S_tM_1,M_3) \le C'e^{-\alpha't}.$$
\end{theorem}

\section{Acknowledgments}
The authors would like to dedicate this work to the memory of Igor D. Chueshov. It is with remarkable sadness that we say goodbye to a truly exceptional mathematician, friend, and collaborator. 

The authors also wish to thank the referee for his or her careful reading which greatly improved the quality of the manuscript. 

The second author was partially supported by the National Science Foundation with grant NSF-DMS-0606682 and the United States Air Force Office of Scientific Research with grant AFOSR-FA99550-9-1-0459. The third author was partially supported by National Science Foundation with grant NSF-DMS-1504697.

\end{document}